\documentclass[12pt,a4paper]{article}
\usepackage[margin=2cm]{geometry}
\usepackage{glossaries}
\usepackage{amsmath}
\usepackage{amssymb}
\usepackage{physics}
\usepackage{enumerate}
\usepackage[utf8]{inputenc}
\usepackage{amsthm}
\usepackage{mathrsfs}
\usepackage{mathtools}

\numberwithin{equation}{section}

\allowdisplaybreaks

\newtheorem{thm}{Theorem}[section]
\newtheorem{lem}[thm]{Lemma}

\theoremstyle{remark}
\newtheorem*{remark}{Remark}

\usepackage[hang]{footmisc}
\title{Eigenvalue asymptotics for a class of multi-variable Hankel matrices}
\date{}
\author{Christos Panagiotis Tantalakis \thanks{christos.tantalakis@kcl.ac.uk} \\ Department of Mathematics, King’s College London}

\begin{document}
	
\maketitle

\begin{abstract}
A one-variable Hankel matrix $H_a$ is an infinite matrix $H_a=[a(i+j)]_{i,j\geq0}$. Similarly, for any $d\geq2$, a $d$-variable Hankel matrix is defined as $H_{\mathbf{a}}=[\mathbf{a}(\mathbf{i}+\mathbf{j})]$, where $\mathbf{i}=(i_1,\dots,i_d)$ and $\mathbf{j}=(j_1,\dots,j_d)$, with $i_1,\dots,i_d,j_1,\dots,j_d\geq0$. For $\gamma>0$, A. Pushnitski and D. Yafaev proved that the eigenvalues of the compact one-variable Hankel matrices $H_a$ with $a(j)=j^{-1}(\log j)^{-\gamma}$, for $j\geq2$, obey the asymptotics $\lambda_n(H_a)\sim C_\gamma n^{-\gamma}$, as $n\to+\infty$, where the constant $C_\gamma$ is calculated explicitly. This paper presents the following $d$-variable analogue. Let $\gamma>0$ and $a(j)=j^{-d}(\log j)^{-\gamma}$, for $j\geq2$. If $\mathbf{a}(j_1,\dots,j_d)=a(j_1+\dots+j_d)$, then $H_{\mathbf{a}}$ is compact and its eigenvalues follow the asymptotics $\lambda_n(H_{\mathbf{a}})\sim C_{d,\gamma}n^{-\gamma}$, as $n\to+\infty$, where the constant $C_{d,\gamma}$ is calculated explicitly.
\end{abstract}

\section{Introduction}

\subsection{One-variable Hankel operators}

Let $\mathbb{N}_0=\{0,1,2,3,\dots\}$. Given a complex valued sequence $a=\{a(j)\}_{j\in\mathbb{N}_0}$, a Hankel operator (matrix) $H_a$ on $\ell^2(\mathbb{N}_0)$ is formally defined by
\begin{equation*}
	\big(H_ax\big)(i)=\sum_{j\in\mathbb{N}_0}a(i+j)x(j), \ \forall i\in\mathbb{N}_0, \ \forall x=\{x(j)\}_{j\in\mathbb{N}_0}\in\ell^2(\mathbb{N}_0).
\end{equation*}
The sequence $a$ is called a \textit{parameter sequence}.

Nehari's theorem \cite[Theorem 1.1]{peller2012hankel} gives a necessary and sufficient condition for $H_a$ to be bounded on $\ell^2(\mathbb{N}_0)$. A simple sufficient condition is given by $a(j)=O(j^{-1})$, when $j\to+\infty$. A sufficient condition for compactness is $a(j)=o(j^{-1})$, when $j\to+\infty$. Note that these two conditions are also necessary in the case of positive Hankel operators \cite[Theorems 3.1, 3.2]{widom1966hankel}.

Let $\alpha>0$ and consider the Hankel operator $H_a$, where
\begin{equation*}
	a(j)=(j+1)^{-\alpha}, \ \forall j\in\mathbb{N}_0.
\end{equation*}
For $\alpha\in(0,1)$, $H_a$ is not bounded. When $\alpha=1$, $H_a$ is bounded but not compact. In this case, $H_a$ is known as Hilbert's matrix. Finally, for $\alpha>1$, $a(j)=o(j^{-1})$, as $j\to+\infty$, and so, $H_a$ is bounded and compact. From this discussion, it is inferred that the exponent $\alpha=1$ is the boundedness-compactness threshold.

In \cite{widom1966hankel} (see the example after Theorem 3.3) it is proved that the eigenvalue asymptotics of the Hankel operators with parameter sequence $a(j)=(j+1)^{-\alpha}$, where $\alpha>1$, are described by
\begin{equation*}
	\lambda_{n}(H_a)=\exp(-\pi \sqrt{2(\alpha-1) n}+o(\sqrt{n})), \ n\to+\infty.
\end{equation*}

In \cite{pushnitski2015asymptotic} A. Pushnitski and D. Yafaev studied a whole class of Hankel operators that lies between the cases $\alpha=1$ and $\alpha>1$. That was achieved by considering parameter sequences $a=\{a(j)\}_{j\in\mathbb{N}_0}$ of the following type
\begin{equation*}
	a(j)=j^{-1}(\log j)^{-\gamma}, \ \forall j\geq 2,
\end{equation*}
where $\gamma>0$. They concluded that if $\{\lambda_n^+(H_a)\}_{n\in\mathbb{N}}$ is the non-increasing sequence of positive eigenvalues of $H_a$, and $\lambda_n^{-}(H_a)=\lambda_n^+(-H_a)$, then the eigenvalues of the corresponding Hankel operator $H_a$ obey the following asymptotic law
\begin{equation}\label{eq0}
	\lambda_n^{+}(H_a)=C_\gamma n^{-\gamma}+o(n^{-\gamma}) \ \ \text{and } \ \lambda_n^-(H_a)=o(n^{-\gamma}), \ n\to+\infty,
\end{equation}
where
\begin{equation}\label{17}
	C_\gamma=\left[\frac{1}{2\pi}\int_\mathbb{R}\left(\frac{\pi}{\cosh(\pi x)}\right)^{\frac{1}{\gamma}}\dd{x}\right]^\gamma=2^{-\gamma}\pi^{1-2\gamma}B\left(\frac{1}{2\gamma},\frac{1}{2}\right)^\gamma;
\end{equation}
here $B(\cdot,\cdot)$ is the Beta function.

\subsection{Multi-variable Hankel operators}

The purpose of this section is to introduce the multi-variable Hankel operators and develop the $d$-variable analogue of the asymptotics (\ref{eq0}).

From now on, we will denote all the multi-variable functions and their arguments by boldface letters. So, for $d\geq 2$ consider the set $\mathbb{N}_0^d=\{\mathbf{j}=(j_1,j_2,\dots,j_d): \ j_i\in\mathbb{N}_0, \ \text{for } i=1,2,\dots,d\}$ and the space $\ell^2(\mathbb{N}^d_0)$ of $d$-variable square summable sequences $\mathbf{x}=\{\mathbf{x}(\mathbf{j})\}_{\mathbf{j}\in\mathbb{N}_0^d}$.
Let $\mathbf{a}=\{\mathbf{a}(\mathbf{j})\}_{\mathbf{j}\in\mathbb{N}_0^d}$ be a complex valued sequence and define, formally, the Hankel operator $H_\mathbf{a}$ on $\ell^2(\mathbb{N}_0^d)$ by
\begin{equation*}
	\big(H_\mathbf{a}\mathbf{x}\big)(\mathbf{i}):=\sum_{\mathbf{j}\in\mathbb{N}_{0}^{d}}\mathbf{a}(\mathbf{i}+\mathbf{j})\mathbf{x}(\mathbf{j}), \ \forall \mathbf{i}\in\mathbb{N}_0^d, \ \forall \mathbf{x}=\{\mathbf{x}(\mathbf{j})\}_{\mathbf{j}\in\mathbb{N}_0^d}\in\ell^2(\mathbb{N}_0^d).
\end{equation*}
The sequence $\mathbf{a}$ is a \textit{parameter sequence}.

To the best of the author's knowledge, no necessary and sufficient conditions for the boundedness or compactness of $H_\mathbf{a}$ are available at present. Heuristically, $\mathbf{a}(\mathbf{j})$ can go to zero at different rates in different directions, which makes the problem more subtle than in the one-variable case. One can make progress by focusing on a subclass of sequences $\mathbf{a}(\mathbf{j})$. In this paper, we consider the following subclass. Let $a=\{a(j)\}_{j\in\mathbb{N}_0}$ be a one-variable sequence and define
\begin{equation*}
	\mathbf{a}(\mathbf{j})=a(\left|\mathbf{j}\right|), \ \text{where } \left|\mathbf{j}\right|=\sum_{i=1}^{d}j_i, \ \forall \mathbf{j}=(j_1,j_2,\dots,j_d)\in\mathbb{N}_0^d.
\end{equation*}
In this case, it can be verified that $H_\mathbf{a}$ is bounded if $a(j)=O(j^{-d})$ and compact if $a(j)=o(j^{-d})$, when $j\to+\infty$. Moreover, for $\alpha>0$ consider the sequence 
\begin{equation*}
	a(j)=(j+1)^{-\alpha}, \ \forall j\in\mathbb{N}_0.
\end{equation*}
If $\alpha\in(0,d)$, then $H_\mathbf{a}$ is unbounded. If $\alpha=d$, $H_\mathbf{a}$ is bounded but not compact and for $\alpha>d$, the aforementioned tests imply boundedness and compactness. Therefore, the boundedness-compactness threshold exponent, for this choice of the parameter sequence $\mathbf{a}$, is $\alpha=d$. The main result of this paper is the $d$-variable analogue of (\ref{eq0}). We first give a simple version of our result, Theorem \ref{thm};
a more complete statement is Theorem \ref{thm0} below. In order to formulate Theorem \ref{thm}, let $\mathcal{F}$ be the Fourier transform on the real line; i.e.
\begin{equation}\label{F}
	\big(\mathcal{F}f\big)(x)=\int_{\mathbb{R}}f(y)e^{-2\pi i xy}\dd{y}, \ \forall x\in\mathbb{R}.
\end{equation}
The inverse Fourier transform, $\mathcal{F}^*f$, of $f$ will be often denoted by $\check{f}$.

\begin{thm}\label{thm}
Let $\gamma>0$ and consider the parameter sequence $\mathbf{a}(\mathbf{j})=a(\left|\mathbf{j}\right|)$, for all $\mathbf{j}\in\mathbb{N}_0^d$, where
\begin{equation}\label{0}
	a(j)=j^{-d}(\log j)^{-\gamma}, \ \forall j\geq2.
\end{equation}
Moreover, for any $j\in\mathbb{N}$, define the function
\begin{equation}\label{eq2}
	\phi_j(x):=\frac{1}{\cosh[j](\frac{x}{2})}, \ \forall x\in\mathbb{R}.
\end{equation}
Then the corresponding Hankel operator $H_\mathbf{a}$ is self-adjoint, compact, and it presents power eigenvalue asymptotics of the form below:
\begin{equation*}
	\lambda_n^+(H_\mathbf{a})=C_{d,\gamma}n^{-\gamma}+o(n^{-\gamma}) \ \ \text{and } \ \lambda_n^-(H_\mathbf{a})=o(n^{-\gamma}), \ n\to+\infty,
\end{equation*}
where
\begin{equation}\label{18}
	C_{d,\gamma}=\frac{1}{2^d(d-1)!}\left( \int_\mathbb{R}\check{\phi}_d^{\frac{1}{\gamma}}(x)\dd{x} \right)^{\gamma}.
\end{equation}
\end{thm}
\begin{remark}
It is worth to notice that relation (\ref{18}) gives (as expected) (\ref{17}) when $d=1$. For observe that $\check{\phi}_1(x)=\frac{2\pi}{\cosh(2\pi^2 x)}$. Then, by applying the change of variables $y=2\pi x$, we get
\begin{equation*}
	C_{1,\gamma}=\frac{1}{2}\left(\int_\mathbb{R}\check{\phi}_1^{\frac{1}{\gamma}}(x)\dd{x}\right)^\gamma=\left[\frac{1}{2\pi}\int_\mathbb{R}\left(\frac{\pi}{\cosh(\pi y)}\right)^{\frac{1}{\gamma}}\dd{y}\right]^\gamma=C_\gamma,
\end{equation*}
where $C_\gamma$ is the constant that is defined by (\ref{17}).
\end{remark}

\subsection{Main result}

A generalisation of Theorem \ref{thm} leads to the main result, Theorem \ref{thm0}. For any $\gamma>0$, define
\begin{equation}\label{8}
	M(\gamma):=\begin{cases}
		0, & \gamma\in(0,\frac{1}{2})\\
		\left[\gamma\right]+1, & \gamma\geq\frac{1}{2}
	\end{cases},
\end{equation}
where $\left[\gamma\right]=\max\{x\in\mathbb{Z}: \ x\leq\gamma\}$. In addition, for any sequence $a=\{a(j)\}_{j\in\mathbb{N}_0}$, define the sequence of iterated differences $a^{(m)}=\{a^{(m)}(j)\}_{j\in\mathbb{N}_0}$, where $m\in\mathbb{N}_0$, with
\begin{equation*}
	a^{(0)}=a \ \ \text{and } \ a^{(m)}(j)=a^{(m-1)}(j+1)-a^{(m-1)}(j), \ \forall j\in\mathbb{N}_0, \ \forall m\in\mathbb{N}.
\end{equation*}

\begin{thm}\label{thm0}
	Let $\gamma>0$, $b_1,b_{-1}\in\mathbb{R}$, and $a$ be a real valued sequence of $\mathbb{N}_0$, such that
	\begin{equation}\label{eq5}
		a(j)=\big(b_1+(-1)^jb_{-1}\big)j^{-d}(\log j)^{-\gamma}+g_1(j)+(-1)^jg_{-1}(j), \ \forall j\geq2,
	\end{equation}
	where both $g_1$ and $g_{-1}$ satisfy the following condition:
	\begin{equation*}
		g_{\pm 1}^{(m)}(j)=o\big(j^{-d-m} (\log j)^{-\gamma} \big), \ j\to+\infty,
	\end{equation*}
	for $m=0,1,\dots,M(\gamma)$. If $H_\mathbf{a}$ is the Hankel operator, where $\mathbf{a}(\mathbf{j})=a\left(|\mathbf{j}|\right), \ \forall \mathbf{j}\in\mathbb{N}_0^d$, then it is a self-adjoint, compact operator and its eigenvalues satisfy the following asymptotic law
	\begin{equation}\label{eq1}
		\lambda^\pm(H_\mathbf{a})=C^\pm n^{-\gamma}+o(n^{-\gamma}), \ n\to+\infty.
	\end{equation}
	The leading term coefficients are given by
	\begin{equation}\label{eq2_12}
		C^{\pm}=\left((b_1)^{\frac{1}{\gamma}}_{\pm}+(b_{-1})^{\frac{1}{\gamma}}_{\pm}\right)^{\gamma}C_{d,\gamma},
	\end{equation}
where $C_{d,\gamma}$ is defined in (\ref{18}) and $(x)_\pm:=\max\{0,\pm x\}$, for any $x\in\mathbb{R}$.
\end{thm}

\subsection{Proof outline}

In order to derive the spectral asymptotics for the class of operators that were introduced in Theorem~\ref{thm0}, we follow the steps that are listed below. In the sequel, we give a brief description of each one of them.
\begin{itemize}
	\item Construction of a model operator (see \S\ref{MO}),
	\item reduction of the model operator to pseudo-differential operators (see \S\ref{PDO}),
	\item use of Weyl-type spectral asymptotics of the respective pseudo-differential operators (see \S\ref{WA}),
	\item reduction of the error terms to one-variable weighted Hankel operators (see \S\ref{RWHO}), and
	\item Schatten class inclusions of the error terms (see \S\ref{SCI}).
\end{itemize}

The construction of the model operator aims to give the leading term in the eigenvalue asymptotics. More precisely, the model operator will be a Hankel operator which behaves ``similarly" to the initial Hankel operator but whose eigenvalue asymptotics are retrieved much easier and explicitly. By examining for simplicity the case of $a$ given by (\ref{0}), the model operator will be a Hankel operator of the form $\tilde{H}:=H_{\tilde{\mathbf{a}}}$, with parameter sequence $\tilde{\mathbf{a}}(\mathbf{j})=\tilde{a}(|\mathbf{j}|)$, for all $\mathbf{j}\in\mathbb{N}_0^d$.
\begin{remark}
	From now on, objects related with the model operator will be declared with the tilde symbol; e.g. $\tilde{H}$, $\tilde{\mathbf{a}}$, $\tilde{a}$, etc.
\end{remark}
\noindent The sequence $\tilde{a}$ will be chosen to be the Laplace transform of a suitable function $w$, i.e.
\begin{equation*}
	\tilde{a}(j)=\big(\mathcal{L}w\big)(j)=\int_{0}^{+\infty}w(\lambda)e^{-\lambda j}\dd{\lambda}, \ \forall j\in\mathbb{N}_0.
\end{equation*}
The function $w$ is chosen in a way such that $\tilde{a}(j)\sim a(j)$, as $j\to+\infty$; i.e. $\frac{\tilde{a}(j)}{a(j)}\to1$, as $j\to+\infty$. The latter is obtained via a lemma for Laplace transform asymptotics.

In the sequel, the spectral analysis of the model operator, $\tilde{H}$, is reduced to that one of a pseudo-differential operator. To see this, consider the inner product
\begin{equation*}
	(\tilde{H}\mathbf{x},\mathbf{y})=\sum_{\mathbf{i},\mathbf{j}\in\mathbb{N}_0^d}\tilde{a}(|\mathbf{i}+\mathbf{j}|)\mathbf{x}(\mathbf{j})\overline{\mathbf{y}}(\mathbf{i}).
\end{equation*}
By using the fact that $\tilde{a}(j)=\big(\mathcal{L}w\big)(j)$, we can swap summation and integration to obtain
\begin{equation*}
	(\tilde{H}\mathbf{x},\mathbf{y})=\int_{0}^{+\infty} \big(L\mathbf{x}\big)(t)\overline{\big(L\mathbf{y}\big)}(t)\dd{t},
\end{equation*}
where $L:\ell^2(\mathbb{N}_0^d)\to L^2(\mathbb{R}_+)$ is given by
\begin{equation*}
	\big(L\mathbf{x}\big)(t)=\sqrt{w(t)}\sum_{\mathbf{j}\in\mathbb{N}_0^d}e^{-|\mathbf{j}|t}\mathbf{x}(\mathbf{j}), \ \forall t\in\mathbb{R}_+, \ \forall \mathbf{x}\in\ell^2(\mathbb{N}_0^d);
\end{equation*}
note that $\mathbb{R}_+:=(0,+\infty)$.
\begin{remark}
	Notice that in order $L$ to be well-defined, $w$ has to be non-negative. For sake of simplicity, in this introduction, we assume that this is true and the general case is addressed properly in the proofs.
\end{remark}
\noindent Therefore, $\tilde{H}$ can be expressed as a product of two operators, $\tilde{H}=L^*L$, and we can apply the following lemma (\cite[\S8.1, Theorem 4]{birman2012spectral}).
\begin{lem}\label{lem}
Let $L$ be a linear bounded operator, defined on a Hilbert space $\mathscr{H}$. Then, the restrictions $L^*L\upharpoonright(\mathrm{Ker}L^*L)^\perp$ and $LL^*\upharpoonright(\mathrm{Ker}LL^*)^\perp$ are unitarily equivalent.
\end{lem}
\noindent\begin{remark}
We will denote this equivalence by $\simeq$; e.g. $L^*L\simeq LL^*$.
\end{remark}
\noindent Thus, $\tilde{H}$ is unitarily equivalent (modulo kernels) to $LL^*:L^2(\mathbb{R}_+)\to L^2(\mathbb{R}_+)$. Finally, by an exponential change of variable, $LL^*$ is proved to be unitarily equivalent (modulo kernels) to a pseudo-differential operator $\beta(X)\alpha(D)\beta(X)$, where $D$ is the differentiation operator in $L^2(\mathbb{R})$, $Df=-if'$, and $X$ is the multiplication operator (in $L^2(\mathbb{R})$) by the function $\mathrm{id}(x)=x$. Then, by exploiting a Weyl-type spectral asymptotics formula for the operator $\beta(X)\alpha(D)\beta(X)$, we retrieve its eigenvalue asymptotics and thus, those of $\tilde{H}$.

\begin{remark}
The technique of considering the inner product $(H_\mathbf{a}\mathbf{x},\mathbf{y})$ and changing the order of summation and integration was also applied by Widom in \cite{widom1966hankel} for one-variable Hankel operators. In order to derive the eigenvalue asymptotics, Widom also applied Lemma \ref{lem}. This yielded the equivalence to the pseudo-differential operator that we would obtain, if we followed the steps that are described above (for $d=1$). The same equivalence, but in greater generality, is also obtained by Yafaev in \cite[Theorem 7.7]{yafaev2017quasi}.
\end{remark}

Finally, the initial Hankel operator, $H_\mathbf{a}$, can be expressed as a sum of operators, $H_\mathbf{a}=\tilde{H}+(H_\mathbf{a}-\tilde{H})$. Having obtained the eigenvalue asymptotics for $\tilde{H}$, the next step is to prove that the spectral contribution of the operator $H_\mathbf{a}-\tilde{H}$ is negligible, compared to that one of $\tilde{H}$. This will be achieved by proving certain Schatten-Lorentz class inclusions for $H_\mathbf{a}-\tilde{H}$. These inclusions depend on the range of the exponent $\gamma$ in (\ref{0}) and are obtained by a combination of interpolation and reduction to one-variable weighted Hankel operators.

\subsection{List of notation}

For the reader's convenience, we close our introduction by summarising the introduced notation.\\ \ \\
\noindent\underline{Set notation}: Let $\mathbb{R}$ be the set of real numbers, $\mathbb{Z}$ the set of integers, $\mathbb{N}=\{1,2,3,\dots\}$ be the set of natural numbers, and $\mathbb{N}_0=\mathbb{N}\cup\{0\}$. In addition, $\mathbb{R}_+=(0,+\infty)$. We denote by $\mathbb{C}$ the set of complex numbers. Then $\mathbb{D}=\{z\in\mathbb{C}: \ \left|z\right|<1\}$ and $\mathbb{T}=\{z\in\mathbb{C}: \ \left|z\right|=1\}$. Moreover, $\mathbb{T}$ can be identified with the interval $[0,1)$, via the map $t\mapsto e^{2\pi i t}$, for all $t\in[0,1)$. For any $d\geq2$, we can define $d$-Cartesian products of the aforementioned sets; e.g. $\mathbb{R}^d=\{\mathbf{x}=(x_1,x_2,\dots,x_d): \ x_i\in\mathbb{R}, \ \text{for } i=1,2,\dots,d\}$.\\ \ \\
\noindent\underline{Dimension notation}: We use the Roman (standard) typeface for one-dimensional/variable objects and boldface letters for $d$-dimensional/variable ones. For example, let $f(x)$ describe a function defined on $\mathbb{R}$ and $\mathbf{a}=\{\mathbf{a}(\mathbf{j})\}_{\mathbf{j}\in\mathbb{N}_0^d}$ be a $d$-variable sequence.\\ \ \\
\noindent\underline{Sequence notation}: We say that two (real valued) sequences $\{a(j)\}_{j\in\mathbb{N}_0}$ and $\{b(j)\}_{j\in\mathbb{N}_0}$ present the same asymptotic behaviour at infinity, and denote by $a(j)\sim b(j)$, as $j\to+\infty$, when $\frac{a(j)}{b(j)}\to1$, as $j\to+\infty$. For a (complex valued) sequence $a=\{a(j)\}_{j\in\mathbb{N}_0}$, define the sequence of iterated differences $a^{(m)}=\{a^{(m)}(j)\}_{j\in\mathbb{N}_0}$, where $m\in\mathbb{N}_0$, with
\begin{equation*}
	a^{(0)}=a \ \ \text{and } \ a^{(m)}(j)=a^{(m-1)}(j+1)-a^{(m-1)}(j), \ \forall j\in\mathbb{N}_0, \ \forall m\in\mathbb{N}.
\end{equation*}

\noindent\underline{Number notation}: For any real number $x$, we define its integer part $\left[x\right]=\max\{m\in\mathbb{Z}: \ m\leq x\}$ and its positive (resp. negative) part $\left(x\right)_+=\max\{0,x\}$ (resp. $(x)_-=\max\{0,-x\}$). Furthermore, let $\left<x\right>=\sqrt{1+x^2}$. For any real numbers $x$ and $y$, we write $x\lesssim y$ when there exists a non-zero number $c$ such that $x\leq cy$. Finally, for any $d\geq2$,
\begin{equation*}
	\left|\mathbf{j}\right|=\sum_{i=1}^d j_i, \ \forall \mathbf{j}=(j_1,j_2,\dots,j_d)\in\mathbb{N}_0^d.
\end{equation*}
\noindent\underline{Fourier transform}: For a function $\phi:\mathbb{T}\to\mathbb{C}$ the sequence of its Fourier coefficients $\{\big(\Phi\phi\big)(n)\}_{n\in\mathbb{Z}}$ is given by
\begin{equation*}
	\big(\Phi\phi\big)(n)=\int_0^1 \phi(t) e^{-2\pi int}\dd{t}, \ \forall n\in\mathbb{Z}.
\end{equation*}
The Fourier transform $\mathcal{F}f$ of a function $f:\mathbb{R}\to\mathbb{C}$ is given by (\ref{F}). We denote by $\mathcal{F}^*$ its inverse and $\check{f}=\mathcal{F}^*f$.\\ \ \\
\noindent\underline{Operator notation}: For any operator $A$, let $A\upharpoonright S$ be the restriction of $A$ to a subset $S$ of its domain. Two operators $A$ and $B$, in general defined on different spaces, will be called unitarily equivalent modulo kernels (write $A\simeq B$), when they have unitarily equivalent non-zero parts. Namely, when there exists a unitary operator $U$ such that
\begin{equation*}
	A\upharpoonright(\mathrm{Ker}A)^{\perp}=U^*B\upharpoonright(\mathrm{Ker}B)^{\perp}U.
\end{equation*}
We denote by $H_\mathbf{a}$ (resp. $H_a$) all the $d$-variable (resp. one-variable) Hankel operators with parameter sequence $\mathbf{a}$ (resp. $a$). Moreover, when $H_{\mathbf{a}}$ has been defined, objects related with the model operator $\tilde{H}$ that corresponds to $H_{\mathbf{a}}$ will be indicated with the tilde symbol; e.g. $\tilde{\mathbf{a}}$ will refer to the parameter sequence of the model operator, so that $\tilde{H}=H_{\tilde{\mathbf{a}}}$. Finally, for weighted Hankel operators, we use the capital gamma; e.g. $\Gamma$, $\Gamma_a^{\alpha,\beta}$, etc. (see \S\ref{WHO} for the relevant definitions).\\ \ \\
\noindent\underline{Eigenvalue notation}: Let $A$ be an operator and $\{\lambda_n^+(A)\}_{n\in\mathbb{N}}$ be the sequence of its positive eigenvalues. Then $\lambda_n^-(A)=\lambda_n^+(-A)$, $\forall n\in\mathbb{N}$.

\section{Preliminaries}

\subsection{Besov classes}\label{BC}

We define Besov classes of analytic functions on the unit circle $\mathbb{T}$. If $C^\infty_c(\mathbb{R})$ is the set of infinitely many times differentiable functions on $\mathbb{R}$, with compact support, let $v$ be a $C^{\infty}_c(\mathbb{R})$ function, such that $\mathrm{supp}(v)=[2^{-1},2]$, $v(1)=1$, and $v([2^{-1},2])=[0,1]$; notice that $v(2^{-1})=v(2)=0$. Then consider a sequence of $C^{\infty}_c(\mathbb{R})$ non-negative valued functions $\{v_n\}_{n\in\mathbb{N}}$, such that,
\begin{equation*}
	v_n(t)=v\left( \frac{t}{2^n} \right), \ \forall n\in\mathbb{N},
\end{equation*}
for any $t\in\mathbb{R}$, and
\begin{equation*}
	\sum_{n\geq 0}v\left(\frac{t}{2^n}\right)=1, \ \forall t\geq 1.
\end{equation*}
Ensuing, define the polynomials
\begin{equation}\label{1'}
	V_0(z)=\overline{z}+1+z, \ \forall z\in\mathbb{T},
\end{equation}
and, for every $n\in\mathbb{N}$,
\begin{equation}\label{2}
	V_n(z)=\sum_{j\in\mathbb{N}}v_n(j)z^j= \sum_{j=2^{n-1}}^{2^{n+1}}v_n(j)z^j, \ \forall z\in\mathbb{T}.
\end{equation}
Then we say that an analytic function $f$ of $\mathbb{T}$ belongs to the Besov class $B^{p}_{q,r}$ if and only if
\begin{equation*}
	\|f\|_{B^{p}_{q,r}}:=\left( \sum_{n\in\mathbb{N}_0}2^{npr}\|f*V_n\|^r_q \right)^{\frac{1}{r}}<+\infty.
\end{equation*}

\noindent The lemma below can be found in \cite[Lemma 4.6]{pushnitski2015sharp}.

\begin{lem}\label{lem_2}
Assume that $\gamma\geq\frac{1}{2}$ and let $M(\gamma)$ be as defined in (\ref{8}). Moreover, let $\{a(j)\}_{j\in\mathbb{N}_{0}}$ be a sequence of complex numbers which satisfies (\ref{eq_3}) and consider the function
\begin{equation*}
	\phi(z)=\sum_{j\in\mathbb{N}_0}a(j)z^j, \ \forall z\in\mathbb{T}.
\end{equation*}
If $V_n$ are as defined in (\ref{2}), then, for every $q>\frac{1}{M(\gamma)}$ and every $n\in\mathbb{N}$ such that $2^{n-1}\geq M(\gamma)$,
\begin{equation}\label{eq3}
	\left\| \phi*V_n \right\|_{\infty}\leq \sum_{j=2^{n-1}}^{2^{n+1}}|a(j)|,
\end{equation}
and
\begin{equation}\label{eq4}
	2^n\left\| \phi*V_n \right\|_q^q\leq C_q \left( \sum_{m=0}^{M(\gamma)}\sum_{j=2^{n-1}-M(\gamma)}^{2^{n+1}}(1+j)^m \left| a^{(m)}(j) \right| \right)^q,
\end{equation}
for some positive constant $C_q$, depending only on $q$.
\end{lem}

\subsection{Schatten classes}

Consider a compact operator $T$ and the sequence of its singular values $\{s_n\}_{n\in\mathbb{N}}$; i.e. the sequence of (positive) eigenvalues of $\sqrt{T^*T}$. Denote by $\mathbf{S}_\infty$ the space of compact operators. For $p\in(0,+\infty)$, we define the Schatten class $\mathbf{S}_p$, the Schatten-Lorentz classes $\mathbf{S}_{p,q}$ and $\mathbf{S}_{p,\infty}$, and the class $\mathbf{S}_{p,\infty}^{0}$ by the following conditions:
\begin{equation*}
	\begin{array}{ccc}
		T\in\mathbf{S}_p & \Leftrightarrow & \left\|T\right\|_{\mathbf{S}_p}:=\left(\displaystyle\sum_{n\in\mathbb{N}}s_n^p\right)^{\frac{1}{p}}<+\infty\\ \ \\
		T\in\mathbf{S}_{p,q} & \Leftrightarrow & \left\|T\right\|_{\mathbf{S}_{p,q}}:=\left(\sum_{n\in\mathbb{N}}\frac{(n^{\frac{1}{p}}s_n)^q}{n}\right)^{\frac{1}{q}}<+\infty\\
		T\in\mathbf{S}_{p,\infty} & \Leftrightarrow & \left\|T\right\|_{\mathbf{S}_{p,\infty}}:=\displaystyle\sup_{n\in\mathbb{N}}n^{\frac{1}{p}}s_n<+\infty\\ \ \\
		T\in\mathbf{S}_{p,\infty}^{0} & \Leftrightarrow & \displaystyle\lim_{n\to+\infty}n^{\frac{1}{p}}s_n=0.
	\end{array}
\end{equation*}
Recall that $\mathbf{S}_{p}\subset\mathbf{S}_{p,\infty}^0\subset\mathbf{S}_{p,\infty},$ with $\|\cdot\|_{\mathbf{S}_{p,\infty}}\leq\|\cdot\|_{\mathbf{S}_p}, \ \forall p>0$.
\noindent Finally, Lemma \ref{lem2_1} \cite[Corollary 2.2]{gohberg1978introduction} suggests that perturbations with $\mathbf{S}_{p,\infty}^0$ operators leave the eigenvalue asymptotics unaffected.

\begin{lem}[K. Fan]\label{lem2_1}
Let $S$ and $T$ be two compact, self-adjoint operators on a Hilbert space. If
\begin{equation*}
	\begin{array}{rclr}
		\lambda_{n}^{\pm}(S)=K^{\pm}n^{-\gamma}+o(n^{-\gamma}), & \text{and} & \ s_n(T)=o(n^{-\gamma}), & \ n\to+\infty,
	\end{array}
\end{equation*}
then
$$\lambda_{n}^{\pm}(S+T)=K^{\pm}n^{-\gamma}+o(n^{-\gamma}),$$
for some constants $K^{\pm}$.
\end{lem}

\noindent For the next lemma, define $C^\infty_c(\mathbb{R}^2)$ to be the set of all infinitely differentiable, compactly supported functions $\varkappa:\mathbb{R}^2\to\mathbb{R}$.

\begin{lem}\label{lem3}
Let $\varkappa\in C^{\infty}_{c}(\mathbb{R}^2)$. Then the integral operator $\mathcal{K}:L^2(\mathbb{R})\to L^2(\mathbb{R})$, with
\begin{equation*}
	\big(\mathcal{K}f\big)(x)=\int_\mathbb{R}f(y)\varkappa(x,y)\dd{y}, \ \forall x\in\mathbb{R},
\end{equation*}
belongs to all Schatten classes $\mathbf{S}_p$, for $p>0$.
\end{lem}

\begin{remark}
	Lemma \ref{lem3} is a minor modification of \cite[Chapter 30.5, Theorem 13]{lax2014functional}. More precisely, Theorem~13 proves the $\mathbf{S}_1$ inclusion of $\mathcal{K}$, but the proof for the rest of $\mathbf{S}_p$ is obtained by the same argument.
\end{remark}

\subsection{Asymptotic orthogonality in $\mathbf{S}_{p,\infty}$}

Let $A$ and $B$ be two operators that belong to the class $\mathbf{S}_{p,\infty}$, where $p>0$. Notice that $B^*A$ and $BA^*$ belong to $\mathbf{S}_{\frac{p}{2},\infty}$. We will call $A$ and $B$ orthogonal if $B^*A=BA^*=0$, and asymptotically orthogonal if $B^*A$ and $BA^*$ belong to $\mathbf{S}_{\frac{p}{2},\infty}^0$.

Asymptotic orthogonality plays an important role when we want to obtain the spectral asymptotics of the operator $A+B$, while we know those of $A$ and $B$. More precisely, for compact, self-adjoint operators, there is the following Lemma, which is a special case of \cite[Theorem 2.3]{pushnitski2016spectral}.

\begin{lem}\label{lem1}
Let $A$ and $B$ be two self-adjoint operators in $\mathbf{S}_{p,\infty}$, for some $p>0$. Assume that the asymptotics of their positive and negative eigenvalues, $\lambda^{\pm}_n(A)$ and $\lambda_n^\pm(B)$, are given by
\begin{equation*}
	\lambda_n^\pm(A)=C^\pm_An^{-\frac{1}{p}}+o(n^{-\frac{1}{p}}), \ n\to+\infty;
\end{equation*}
and
\begin{equation*}
	\lambda_n^\pm(B)=C^\pm_Bn^{-\frac{1}{p}}+o(n^{-\frac{1}{p}}), \ n\to+\infty.
\end{equation*}
If $A$ and $B$ are asymptotically orthogonal, then
\begin{equation*}
	\lambda_n^\pm(A+B)=\big((C_A^\pm)^p+(C_B^\pm)^p\big)^{\frac{1}{p}}n^{-\frac{1}{p}}+o(n^{-\frac{1}{p}}), \ n\to+\infty.
\end{equation*}
\end{lem}

\subsection{Weighted Hankel operators}\label{WHO}

Let $\{w_1(j)\}_{j\in\mathbb{N}_0}$, $\{w_2(j)\}_{j\in\mathbb{N}_0}$ and $a=\{a(j)\}_{j\in\mathbb{N}_0}$ be three complex valued sequences and define, formally, the operator $\Gamma:\ell^2(\mathbb{N}_0)\to\ell^2(\mathbb{N}_0)$ by $\Gamma=\mathcal{M}_{w_1}H_a\mathcal{M}_{w_2}$, where $\mathcal{M}_{w}$ is the multiplication operator by a sequence $w=\{w(j)\}_{j\in\mathbb{N}_0}$. In addition, for any $\alpha$, $\beta>0$, define the special class of weighted Hankel operators $\Gamma_a^{\alpha,\beta}=\mathcal{M}_{w_1}H_a\mathcal{M}_{w_2}$, where $w_1(j)=(j+1)^{\alpha}$ and $w_2(j)=(j+1)^{\beta}$, for all $j\in\mathbb{N}_0$. A Schatten class criterion for this class of weighted operators is given by the Theorem \ref{thm2_8} \cite[Theorem B]{aleksandrov2004distorted}.

\begin{thm}\label{thm2_8}
	Let $p\in(0,+\infty)$, $\alpha$, $\beta>0$, and $\phi$ be an analytic function on $\mathbb{T}$ and $\Phi\phi$ the sequence of its Fourier coefficients. Then
	\begin{equation*}
		\|\phi\|_{B_{p}^{\frac{1}{p}+\alpha+\beta}} \lesssim \|\Gamma_{\Phi\phi}^{\alpha,\beta}\|_{\mathbf{S}_p}\lesssim \|\phi\|_{B_{p}^{\frac{1}{p}+\alpha+\beta}},
	\end{equation*}
	where $\Gamma_{\Phi\phi}^{\alpha,\beta}$ is the weighted Hankel operator described by the matrix $[(i+1)^{\alpha}(\Phi\phi)(i+j)(j+~1)^{\beta}]_{i,j\geq0}$.
\end{thm}
\noindent The following lemma is a combination of Theorem \ref{thm2_8} and \cite[Theorem 6.4.4]{peller2012hankel}. The reader can find a sketch of proof in the Appendix \ref{A}.

\begin{lem}\label{lem0}
	Define the measure space
	\begin{equation*}
		(\mathcal{M},\mu):=\bigoplus_{n\in\mathbb{N}_0}\left( \mathbb{T}, 2^{n}\mathbf{m} \right),
	\end{equation*}
	where $\mathbf{m}$ is the Lebesgue measure on $\mathbb{T}$. Let $p\in(0,+\infty)$, $q\in(0,+\infty]$ and $\mathcal{B}_{p,q}^{\frac{1}{p}+d-1}$ be the space of analytic functions $\phi$ on $\mathbb{D}=\{z\in\mathbb{C}: \ |z|<1\}$ such that
	\begin{equation}\label{4}
		\bigoplus_{n\in\mathbb{N}_0}2^{n(d-1)}\phi*V_n\in L^{p,q}(\mathcal{M},\mu),
	\end{equation}
	where the polynomials $V_n$ are defined in (\ref{1'}) and (\ref{2}). Then
	\begin{equation*}
		\left\|\Gamma_{\Phi\phi}^{\frac{d-1}{2},\frac{d-1}{2}}\right\|_{\mathbf{S}_{p,q}}\lesssim\left\| \phi \right\|_{\mathcal{B}_{p,q}^{\frac{1}{p}+d-1}},
	\end{equation*}
where $\Phi\phi$ is the sequence of the Fourier coefficients of $\phi$.
\end{lem}

\subsection{Laplace transform estimates}

Let $\mathcal{L}:L^2(\mathbb{R}_+)\to L^2(\mathbb{R}_+)$ be the Laplace transform, given by
\begin{equation*}
	\big(\mathcal{L}f\big)(t)=\int_0^{+\infty}f(\lambda)e^{-\lambda t}\dd{\lambda}, \ \forall t>0, \ \forall f\in L^2(\mathbb{R}_+).
\end{equation*}
In this paper, we are interested in the $t\to+\infty$ asymptotic behaviour of the Laplace transform of functions with logarithmic singularities near zero of the form $f(\lambda)=\lambda^n\left|\log\lambda\right|^{-\gamma}$, for $\gamma>0$. This asymptotic behaviour is obtained by Lemma \ref{lem2_3} \cite[Lemma 3.3]{pushnitski2015asymptotic}.

\begin{lem}\label{lem2_3}
	Let
	$$I_n(t)=\int_{0}^{\lambda_0}\lambda^n|\log\lambda|^{-\gamma} e^{-\lambda t}\dd{\lambda},$$
	where $\gamma>0$, $n\in\mathbb{N}_{0}$ and $\lambda_0\in(0,1)$. Then
	$$I_n(t)=n! \ t^{-1-n}|\log t|^{-\gamma}\left( 1+O(|\log t|^{-1}) \right), \ t\to+\infty.$$
\end{lem}

\subsection{Weyl-type spectral asymptotics for pseudo-differential operators}

Let $X$ and $D$ be, respectively, the multiplication and the differentiation operator in $L^2(\mathbb{R})$. They are self-adjoint operators, defined on appropriate domains, and given by
\begin{equation*}
	\big(Xf\big)(x)=xf(x), \ \text{and } \big(Df\big)(x)=-if'(x).
\end{equation*}
The following lemma (\cite[Theorem 2.4]{pushnitski2015asymptotic}) deals with pseudo-differential operators of the form $\Psi=\beta(X)\alpha(D)\beta(X)$. Notice that $\alpha(D)=\mathcal{F}^*\alpha(2\pi X)\mathcal{F}$, an expression which will prove to be useful in the sequel.

\begin{lem}\label{lem2_2}
	Let $\alpha$ be a real valued function in $C^{\infty}(\mathbb{R})$, such that
	\[ \alpha(x)=\left\{
	\begin{array}{rl}
		\alpha(+\infty)x^{-\gamma} + o(x^{-\gamma}), & x\to+\infty\\ \ \\
		\alpha(-\infty)|x|^{-\gamma} + o(x^{-\gamma}), & x\to-\infty,
	\end{array}
	\right. \]
	for some real constants $\alpha(+\infty)$, $\alpha(-\infty)$ and $\gamma>0$. Now let $\beta$ be a real valued function on $\mathbb{R}$ such that
	$$|\beta(x)|\leq M \left< x \right>^{-s}, \ \forall x\in\mathbb{R},$$
	where $s>\frac{\gamma}{2}$ and $M$ is a non-negative constant. Define the pseudo-differential operator $\Psi=\beta(X)\alpha(D)\beta(X)$ on $L^2(\mathbb{R})$. Then $\Psi$ is compact and obeys the following eigenvalue asymptotic formula:
	$$\lambda_{n}^{\pm}=C^{\pm}n^{-\gamma}+o(n^{-\gamma}), \ n\to+\infty,$$
	where
	\begin{equation*}
		C^{\pm}=\left[ \frac{1}{2\pi} \left( \alpha(+\infty)_{\pm}^{\frac{1}{\gamma}} + \alpha(-\infty)_{\pm}^{\frac{1}{\gamma}} \right) \int_{\mathbb{R}} |\beta(x)|^{\frac{2}{\gamma}}\dd{x} \right]^{\gamma}.
	\end{equation*}	
\end{lem}
\noindent Above $\left<x\right>:=\sqrt{1+x^2}, \forall x\in\mathbb{R}$.

\section{Construction of the model operator}\label{MO}

Consider the cut-off function $\chi_0\in C^\infty(\mathbb{R})$ such that
\begin{equation}\label{eqn0'}
	\mathcal{\chi}_{0}(t)=\left\{
	\begin{array}{rl}
		1, & 0< t\leq\frac{1}{2}\\ \ \\
		0, & t\geq \frac{3}{4}
	\end{array}, \right.
\end{equation}
and $0\leq\mathcal{\chi}_{0}\leq1$. Let $\gamma>0$ and define the function
\begin{equation}\label{eq2_11'}
	w(t)=\frac{1}{(d-1)!}t^{d-1}\left|\log t\right|^{-\gamma}\chi_0(t), \ \forall t>0.
\end{equation}
If
\begin{equation*}
	\big(\mathcal{L}w\big)(t)=\int_{0}^{+\infty}w(\lambda)e^{-\lambda t}\dd{\lambda}, \ \forall t>0,
\end{equation*}
let $b_1,b_{-1}\in\mathbb{R}$ and define the sequence $\tilde{a}=\{\tilde{a}(j)\}_{j\in\mathbb{N}}$ by
\begin{equation}\label{eq2_4'}
	\tilde{a}(j)=b_1\big(\mathcal{L}w\big)(j)+(-1)^jb_{-1}\big(\mathcal{L}w\big)(j), \ \forall j\in\mathbb{N}.
\end{equation}
Then we define the model operator $\tilde{H}:=H_{\tilde{\mathbf{a}}}$, with parameter sequence $\tilde{\mathbf{a}}(\mathbf{j})=\tilde{a}(|\mathbf{j}|)$, $\forall \mathbf{j}\in\mathbb{N}_0^d$. For the sequence $\tilde{a}$, we have the following lemma.

\begin{lem}\label{lem3''}
	Let $w$ be the function described in (\ref{eq2_11'}) and $\tilde{a}$ be the sequence defined in (\ref{eq2_4'}). Then  $\tilde{a}$ satisfies the following formula:
	\begin{equation}\label{0''''}
		\tilde{a}(j)=\left(b_1+(-1)^jb_{-1}\right)j^{-d}(\log j)^{-\gamma}+\tilde{g}_1(j)+(-1)^j\tilde{g}_{-1}(j), \ \forall j\geq2,
	\end{equation}
	where the error sequences $\tilde{g}_{\pm 1}$ present the following asymptotic behaviour:
	\begin{equation}\label{eqc_18'}
		\tilde{g}^{(m)}_{\pm1}(j)=O\left( j^{-d-m}(\log j)^{-\gamma-1} \right), \ j\to+\infty,
	\end{equation}
	for all $m\in\mathbb{N}_0$.
\end{lem}

\begin{proof}
	First assume that $b_{-1}=0$ and $b_1\neq0$. Then
	\begin{equation*}
		\tilde{a}(j)=b_1\big(\mathcal{L}w\big)(j), \ \forall j\in\mathbb{N}_0,
	\end{equation*}
	and we aim to prove that
	\begin{equation}\label{0'}
		\tilde{a}(j)=b_1j^{-d}\left(\log j\right)^{-\gamma}+\tilde{g}_1(j), \ \forall j\geq 2,
	\end{equation}
	where the error term $\tilde{g}_1$ satisfies (\ref{eqc_18'}). Moreover, without loss of generality, assume that $b_1=1$, otherwise work with $\tfrac{\tilde{a}}{b_1}$. Let $\tilde{g}_1$ be the function below
	\begin{equation}\label{0''}
		\tilde{g}_1(t)=\frac{1}{(d-1)!}\int\limits_{0}^{+\infty}\lambda^{d-1}|\log\lambda|^{-\gamma}\chi_0(\lambda)e^{-\lambda t}\dd{\lambda}-t^{-d}|\log t|^{-\gamma}, \ \forall t>1,
	\end{equation}
	and notice that $\tilde{g}_1\in C^{\infty}(1,+\infty)$. More precisely, for every $m\in\mathbb{N}$ and any $t>1$,
	\begin{multline}\label{lll'}
		\tilde{g}_1^{(m)}(t)=\frac{(-1)^m}{(d-1)!}\int\limits_{0}^{+\infty}\lambda^{d+m-1}|\log\lambda|^{-\gamma}\chi_0(\lambda)e^{-\lambda t}\dd{\lambda}\\
		-\sum_{n=0}^m{m \choose n}\left(\frac{\dd^nt^{-d}}{\dd{t^n}}\right)\left(\frac{\dd^{m-n}(\log t)^{-\gamma}}{\dd{t^{m-n}}}\right).
	\end{multline}
	Moreover, for every $m\in\mathbb{N}_0$ and any $t>0$,
	\begin{multline*}
		\int\limits_{0}^{+\infty}\lambda^{d+m-1}|\log\lambda|^{-\gamma}\chi_0(\lambda)e^{-\lambda t}\dd{\lambda}=\int\limits_{0}^{\frac{1}{2}}\lambda^{d+m-1}|\log\lambda|^{-\gamma}e^{-\lambda t}\dd{\lambda}\\
		+\int\limits_{\frac{1}{2}}^{\frac{3}{4}}\lambda^{d+m-1}|\log\lambda|^{-\gamma}e^{-\lambda t}\chi_0(\lambda)\dd{\lambda}.
	\end{multline*}
	Notice that the second integral converges to zero exponentially fast when $t\to+\infty$. Thus Lemma \ref{lem2_3} yields
	\begin{equation}\label{l'}
		\int\limits_{0}^{+\infty}\lambda^{d+m-1}|\log\lambda|^{-\gamma}\chi_0(\lambda)e^{-\lambda t}\dd{\lambda}=(d+m-1)!\,t^{-d-m}(\log t)^{-\gamma}\left(1+O\left((\log t)^{-1}\right)\right),
	\end{equation}
	when $t\to+\infty$. Besides, notice that, for every $k\in\mathbb{N}$,
	\begin{equation*}
		\frac{\dd^k}{\dd{t^k}}(\log t)^{-\gamma}=O\left(t^{-k}(\log t)^{-\gamma-1}\right), \ t\to+\infty.
	\end{equation*}
	Thus, it is easily verified that,
	\begin{multline}\label{ll'}
		\sum_{n=0}^m{m \choose n}\left(\frac{\dd^nt^{-d}}{\dd{t^n}}\right)\left(\frac{\dd^{m-n}(\log t)^{-\gamma}}{\dd{t^{m-n}}}\right)=\frac{(-1)^m}{(d-1)!}(d+m-1)!\,t^{-d-m}(\log t)^{-\gamma}\\
		+O\left(t^{-d-m}(\log t)^{-\gamma-1}\right), \ \text{when } t\to+\infty.
	\end{multline}
	Then by putting (\ref{l'}) and (\ref{ll'}) back to (\ref{lll'}), we obtain that for every $m\in\mathbb{N}_0$,
	\begin{multline*}
		\tilde{g}_1^{(m)}(t)=\frac{(-1)^m}{(d-1)!}(d+m-1)!\,t^{-d-m}(\log t)^{-\gamma}\left(1+O\left((\log t)^{-1}\right)\right)\\
		-\frac{(-1)^m}{(d-1)!}(d+m-1)!\,t^{-d-m}(\log t)^{-\gamma}+O\left(t^{-d-m}(\log t)^{-\gamma-1}\right), \ \text{for } t\to+\infty.
	\end{multline*}
	Therefore, $\tilde{g}_1(t)$ satisfies the following smoothness property:
	\begin{equation}\label{lll}
		\tilde{g}_1^{(m)}(t)=O\left(t^{-d-m}(\log t)^{-\gamma-1}\right), \ \text{for } t\to+\infty, \ \forall m\in\mathbb{N}_0.
	\end{equation}
	In addition, by (\ref{0''}), the function $\tilde{a}(t):=b_1\big(\mathcal{L}w\big)(t)$, $\forall t>0$, satisfies
	\begin{equation*}
		\tilde{a}(t)=t^{-d}\left|\log t\right|^{-\gamma}+\tilde{g}_1(t), \ \forall t>1.
	\end{equation*}
	Thus, by restricting $\tilde{a}$ on the set of integers greater than or equal to $2$, we get (\ref{0'}). The relation (\ref{eqc_18'}) for $\tilde{g}_1(j)$ is obtained by noticing that $\{\tilde{g}_1(j)\}_{j\geq2}$ is the restriction of the function $\tilde{g}_1$ on the set of integers greater than one, so
	\begin{equation*}
		\tilde{g}_1(j)=O(j^{-d}\left|\log j\right|^{-\gamma-1}), \ t\to+\infty,
	\end{equation*}
	and also
	\begin{equation*}
		\tilde{g}_1^{(m)}(j)=\int_{0}^1\int_{0}^1\dots\int_{0}^1\tilde{g}^{(m)}_1(j+t_1+t_2+\dots+t_m)\dd{t_m}\dots\dd{t_2}\dd{t_1}, \ \forall j\geq 2, \ \forall m\in\mathbb{N},
	\end{equation*}
	where $\tilde{g}_1(t)$ satisfies (\ref{lll}). As a result,
	\begin{equation*}
		\tilde{g}_1^{(m)}(j)=O(j^{-d-m}\left|\log j\right|^{-\gamma-1}), \ t\to+\infty,
	\end{equation*}
	for every $m\in\mathbb{N}$.
	
	\noindent Finally, by repeating the same arguments when $b_1=0$ and $b_{-1}\neq0$, we obtain that
	\begin{equation}\label{0'''}
		\tilde{a}(j)=(-1)^jb_{-1}j^{-d}\left(\log j\right)^{-\gamma}+(-1)^j\tilde{g}_{-1}(j), \ \forall j\geq 2,
	\end{equation}
	where the error term $\tilde{g}_{-1}$ satisfies (\ref{eqc_18'}). By combining (\ref{0'}) and (\ref{0'''}) together we eventually obtain (\ref{0''''}).
\end{proof}

\section{Reduction to pseudo-differential operators}\label{PDO}

Let $\tilde{a}$ (see (\ref{eq2_4'})) be the parameter sequence of the model operator $\tilde{H}$. Then
\begin{equation*}
	\tilde{a}(j)=\tilde{a}_1(j)+\tilde{a}_{-1}(j),
\end{equation*}
where
\begin{equation}\label{1''}
	\tilde{a}_{\pm1}(j)=(\pm1)^{j}b_{\pm1}\big(\mathcal{L}w\big)(j), \ \forall j\in\mathbb{N}_0,
\end{equation}
and $w$ is defined in (\ref{eq2_11'}). Then $\tilde{a}_{1}$ (resp. $\tilde{a}_{-1}$) defines the Hankel operator $\tilde{H}_1$ (resp. $\tilde{H}_{-1}$), with parameter sequence $\tilde{a}_1(|\mathbf{j}|)$, for all $\mathbf{j}\in\mathbb{N}_0^d$ (resp. $\tilde{a}_{-1}(|\mathbf{j}|)$). Thus, $\tilde{H}=\tilde{H}_{1}+\tilde{H}_{-1}$. We reduce the spectral analysis of $\tilde{H}_{\pm1}$ to that of some pseudo-differential operators $\Psi_{\pm1}$.

\begin{lem}\label{R}
	For $j=1,2,\dots,d-1$, let $R_j:L^2(\mathbb{R}_+)\to L^2(\mathbb{R}_+)$ be the integral operator
	\begin{equation*}
		\big(R_jf\big)(t)=\int_0^{+\infty}\sqrt{w(t)}\frac{f(s)}{(s+t)^j}\sqrt{w(s)}\dd{s}, \ \forall t>0, \ \forall f\in L^2(\mathbb{R}_+).
	\end{equation*}
	Then $R_j\in\bigcap_{p>0}\mathbf{S}_p$, for all $j=1,2,\dots,d-1$.
\end{lem}

\begin{proof}
	Let $U:L^2(\mathbb{R}_+)\to L^2(\mathbb{R})$ be the unitary transformation that is given by
	\begin{equation}\label{13}
		\big(Uf\big)(x)=e^{\frac{x}{2}}f(e^x), \ \forall x\in\mathbb{R}, \ \forall f\in L^2(\mathbb{R}_+).
	\end{equation}
	Therefore, by applying the change of variable $s=e^y$ and setting $x=\log t$,
	\begin{equation*}
		\big(R_jf\big)(e^x)=\int_{\mathbb{R}}\sqrt{w(e^x)}\frac{f(e^y)e^y}{(e^x+e^y)^j}\sqrt{w(e^y)}\dd{y}, \ \forall x\in\mathbb{R}, \ \forall f\in L^2(\mathbb{R}_+).
	\end{equation*}
	Moreover, observe that $e^x+e^y=2e^{\frac{x+y}{2}}\cosh(\frac{x-y}{2})$, so that, for any $f\in L^2(\mathbb{R}_+)$,
	\begin{equation*}
		\big(UR_jf\big)(x)=\int_{\mathbb{R}}\sqrt{2^{-d}e^{-(j-1)x}w(e^x)}\frac{\big(Uf\big)(y)}{\cosh[j](\frac{x-y}{2})}\sqrt{2^{-d}e^{-(j-1)y}w(e^y)}\dd{y}, \ \forall x\in\mathbb{R},
	\end{equation*}
for all $j=1,2,\dots,d-1$. For any $j=1,2,\dots,d-1$, define the functions
\begin{equation}\label{llll}
	\alpha_j(x)=2^{-d}e^{-(j-1)x}w(e^x), \ \forall x\in\mathbb{R}.
\end{equation}
Then
\begin{equation*}
	R_j=U^*\alpha_j^{1/2}(X)T_j\alpha_j^{1/2}(X)U,
\end{equation*}
where the operator $T_j:L^2(\mathbb{R})\to L^2(\mathbb{R})$ is the convolution operator with the function $\phi_j$; i.e. for any $j\in\mathbb{N}$,
\begin{equation}\label{l"}
	\big(T_jf\big)(x)=\big(\phi_j*f\big)(x), \ \forall x\in\mathbb{R}, \ \forall f\in L^2(\mathbb{R}),
\end{equation}
where $\phi_j$ is given by (\ref{eq2}). In order $R_j$ to belong to a Schatten class $\mathbf{S}_p$, it is enough to prove that $\alpha_j^{1/2}(X)T_j\alpha_j^{1/2}(X)\in\mathbf{S}_p$, for $j=1,2,\dots,d-1$. To see this, observe that $T_j=\mathcal{F}\beta_j^2(X)\mathcal{F}^*$, where
\begin{equation}\label{l''}
	\beta_j(x)=\sqrt{\check{\phi}_j(x)}, \ \forall x\in\mathbb{R}, \ \forall j\in\mathbb{N}.
\end{equation}
Note that $\check{\phi}_1(x)=2\pi(\cosh(2\pi^2x))^{-1}$, and the latter is positive for any $x\in\mathbb{R}$. Since the convolution of positive functions is positive, $\check{\phi}_j>0$, and thus, $\beta_j$ is well-defined. Then
$$\alpha_j^{1/2}(X)T_j\alpha_j^{1/2}(X)=\alpha_j^{1/2}(X)\mathcal{F}\beta_j^2(X)\mathcal{F}^*\alpha_j^{1/2}(X),$$
and Lemma \ref{lem} implies that the latter is unitarily equivalent (modulo kernels) to the pseudo-differential operator $\beta_j(X)\alpha_j(\tfrac{1}{2\pi}D)\beta_j(X)$. Moreover, (\ref{llll}) implies that
\begin{equation*}
	\alpha_j(x)=\begin{cases}
		0, & \text{when } x\to+\infty\\
		\frac{1}{2^d(d-1)!}e^{-(d-j)|x|}\left|x\right|^{-\gamma}, & \text{when } x\to-\infty
	\end{cases}, \ \forall j=1,2,\dots,d-1.
\end{equation*}
Since $\alpha_j(x)$ decays exponentially fast, when $x\to-\infty$, Lemma \ref{lem2_2} indicates that the pseudo-differential operator $\beta_j(X)\alpha(\tfrac{1}{2\pi}D)\beta_j(X)$ and thus, $\alpha_j^{1/2}(X)T_j\alpha_j^{1/2}(X)$, belong to $\bigcap_{p>0}\mathbf{S}_p$, for all $j=1,2,\dots,d~-~1$.
\end{proof}

\begin{lem}\label{lem2'}
	Let $\tilde{H}_1$ and $\tilde{H}_{-1}$ be the Hankel operators that were defined at the start of Section \ref{PDO}, with parameter sequences $\tilde{a}_1(|\mathbf{j}|)$ and $\tilde{a}_{-1}(|\mathbf{j}|)$, for all $\mathbf{j}\in\mathbb{N}_0^d$, respectively; where $\tilde{a}_{\pm1}$ have been defined in (\ref{1''}). Then there exist two couples of operators $S_1$, $S_{-1}$ and $E_1$, $E_{-1}$, defined on $L^2(\mathbb{R}_+)$, such that
	\begin{enumerate}[(i)]
		\item $\tilde{H}_{\pm1}$ is unitarily equivalent (modulo kernels) to $S_{\pm1}$,
		\item $E_{\pm1}\in\bigcap_{p>0}\mathbf{S}_p$, and
		\item $S_{\pm1}-E_{\pm1}$ is unitarily equivalent (modulo kernels) to a pseudo-differential operator $\Psi_{\pm1}:L^2(\mathbb{R})\to L^2(\mathbb{R})$. More precisely, $\Psi_{\pm1}=\beta(X)\alpha_{\pm1}(\frac{1}{2\pi}D)\beta(X)$, where
		\begin{equation}\label{L17}
			\alpha_{\pm1}(x)=2^{-d}e^{-(d-1)x}b_{\pm1}w(e^x), \ \ \beta(x)=\sqrt{\check{\phi}_d(x)}, \ \forall x\in\mathbb{R},
		\end{equation}
		and $\phi_d$ is defined in (\ref{eq2}).
	\end{enumerate}
\end{lem}

\begin{proof}
First of all, notice that in Lemma \ref{R}, we proved that $\beta$ is well defined, since $\beta=\beta_d$, where the latter is given by (\ref{l''}). We prove the assertion for $\tilde{H}_1$ and the proof for $\tilde{H}_{-1}$ is completely analogous. Moreover, we can assume that $b_1=1$, otherwise work with $\frac{1}{b_1}\tilde{H}_1$.\\ \ \\
	\noindent\textit{(i)} Let $\mathbf{x},\mathbf{y}\in\ell^2(\mathbb{N}_0^d)$. Then
	\begin{equation*}
		\begin{split}
			(\tilde{H}_1\mathbf{x},\mathbf{y}) & = \sum_{\mathbf{i},\mathbf{j}\in\mathbb{N}_0^d} \tilde{a}_1(|\mathbf{i}+\mathbf{j}|)\mathbf{x}(\mathbf{j})\overline{\mathbf{y}}(\mathbf{i})\\
			& = \sum_{\mathbf{i},\mathbf{j}\in\mathbb{N}_0^d}\int_{0}^{+\infty}w(t)e^{-(|\mathbf{i}+\mathbf{j}|)t}\dd{t}\mathbf{x}(\mathbf{j})\overline{\mathbf{y}}(\mathbf{i})\\
			& = (L_1\mathbf{x},L_1\mathbf{y}),
		\end{split}
	\end{equation*}
	where $L_1:\ell^2(\mathbb{N}_0^d)\to L^2(\mathbb{R}_+)$ is defined by
	\begin{equation}\label{14'}
		\big(L_1\mathbf{x}\big)(t)=\sqrt{w(t)}\sum_{\mathbf{j}\in\mathbb{N}_0^d}e^{-|\mathbf{j}|t}\mathbf{x}(\mathbf{j}), \ \forall t\in\mathbb{R}_+, \ \forall \mathbf{x}\in\ell^2(\mathbb{N}_0^d).
	\end{equation}
	Notice that the interchange of summation and integration is justified by the uniform convergence of $\sum_{\mathbf{j}\in\mathbb{N}_0^d}e^{-|\mathbf{j}|t}$, in $\mathbb{R}_+$. Therefore, $\tilde{H}_1=L^*_1L_1$. Moreover, it is not difficult to verify that the formula for the adjoint operator $L_1^*:L^2(\mathbb{R}_+)\to\ell^2(\mathbb{N}_0^d)$ is the following:
	\begin{equation}\label{15'}
		\big(L^*_1f\big)(\mathbf{j})=\int_0^{+\infty}\sqrt{w(t)}f(t)e^{-|\mathbf{j}|t}\dd{t}, \ \forall \mathbf{j}\in\mathbb{N}_0^d, \ \forall f\in L^2(\mathbb{R}_+).
	\end{equation}
	In addition, Lemma \ref{lem} implies that the non-zero parts of $\tilde{H}_1$ and $S_1:=L_1L_1^*$ are unitarily equivalent. Now observe that $S_1:L^2(\mathbb{R}_+)\to L^2(\mathbb{R}_+)$ and
	\begin{equation}\label{16'}
		\begin{split}
			\big(S_1f\big)(t) & = \sqrt{w(t)} \sum_{\mathbf{j}\in\mathbb{N}_{0}^{d}}\int_{0}^{+\infty} f(s) \sqrt{w(s)} e^{-(t+s)|\mathbf{j}|}\dd{s}\\
			& = \int_{0}^{+\infty} \sqrt{w(t)}\frac{f(s)}{\left( 1-e^{-(s+t)} \right)^d}\sqrt{w(s)}\dd{s}, \ \forall t\in\mathbb{R}_+, \ \forall f\in L^2(\mathbb{R}_+).
		\end{split}
	\end{equation}
	\begin{remark}
		Observe that the respective formulae for $L_{-1}$ and $L_{-1}^*$, assuming that $b_{-1}=1$, will be
		\begin{equation}\label{14''}
			\big(L_{-1}\mathbf{x}\big)(t)=\sqrt{w(t)}\sum_{\mathbf{j}\in\mathbb{N}_0^d}(-1)^{|\mathbf{j}|}e^{-|\mathbf{j}|t}\mathbf{x}(j), \ \forall t\in\mathbb{R}_+, \ \forall \mathbf{x}\in\ell^2(\mathbb{N}_0^d),
		\end{equation}
		and
		\begin{equation}\label{15''}
			\big(L^*_{-1}f\big)(\mathbf{j})=(-1)^{|\mathbf{j}|}\int_0^{+\infty}\sqrt{w(t)}f(t)e^{-|\mathbf{j}|t}\dd{t}, \ \forall \mathbf{j}\in\mathbb{N}_0^d, \ \forall f\in L^2(\mathbb{R}_+),
		\end{equation}
		so that $S_{-1}=S_1$.
	\end{remark}
	\ \\ \
	\noindent\textit{(ii)} Observe the formula (\ref{16'}) for $S_1$ and that
	\begin{equation*}
		\frac{1}{\left( 1-e^{-(s+t)} \right)^d}=\frac{1}{(s+t)^d}+\rho(s+t),
	\end{equation*}
	where $\rho$ is real analytic with a pole of order $d-1$ at $0$. Now define the operator $E_1:L^2(\mathbb{R}_+)\to ~L^2(\mathbb{R}_+)$, with
	\begin{equation}\label{E}
		\big(E_1f\big)(t)=\int_{0}^{+\infty} \sqrt{w(t)}f(s)\rho(s+t)\sqrt{w(s)}\dd{s}, \ \forall t\in\mathbb{R}_+, \ \forall f\in L^2(\mathbb{R}_+).
	\end{equation}
	The function $\rho$ can be written as
	\begin{equation*}
		\rho(t)= \sum_{j=1}^{d-1}\frac{c_{-j}}{t^j}+\rho_{\text{an}}(t), \ \forall t\neq0,
	\end{equation*}
	where $\rho_{\text{an}}$ is real analytic and $c_{-j}$ are real constants. Now notice that the function $\sqrt{\chi_0(t)}\rho_{\text{an}}(s+t)\sqrt{\chi_0(s)}$, where $\chi_0$ is defined in (\ref{eqn0'}), belongs to $C^\infty_c(\mathbb{R}^2)$. Then, according to Lemma \ref{lem3}, the integral operator with kernel $\sqrt{\chi_0(t)}\rho_{\text{an}}(s+t)\sqrt{\chi_0(s)}$, belongs to any Schatten class $\mathbf{S}_p$. Moreover, the function $t^{d-1}|\log t|^{-\gamma}$ is bounded near $0$, so that the integral operator with kernel $\sqrt{w(t)}\rho_{\mathrm{an}}(s+t)\sqrt{w(s)}$ belongs to any Schatten class $\mathbf{S}_p$. It remains to prove the same for the integral operators $R_j$ with kernel $\sqrt{w(t)}(s+t)^{-j}\sqrt{w(s)}$, where $j=1,2,\dots,d-1$, which holds true due to Lemma \ref{R}.\\ \ \\
	\noindent\textit{(iii)} By recalling the definitions of $S_1$ in (\ref{16'}) and $E_1$ in (\ref{E}), $S_1-E_1$ is also an operator on $L^2(\mathbb{R}_+)$, described by
	\begin{equation*}
		(S_1-E_1)f(t)=\int_{0}^{+\infty} \sqrt{w(t)}\frac{f(s)}{(s+t)^d}\sqrt{w(s)}\dd{s}, \ \forall t\in\mathbb{R}_+, \ \forall f\in L^2(\mathbb{R}_+).
	\end{equation*}
	Let $U:L^2(\mathbb{R}_+)\to L^2(\mathbb{R})$ be the unitary transformation that was defined in (\ref{13}). Then, by applying the change of variable $s=e^y$ and setting $x=\log t$,
	\begin{equation*}
		(S_1-E_1)f(e^x)=\int_{\mathbb{R}}\sqrt{w(e^x)}\frac{f(e^y)e^y}{(e^x+e^y)^d}\sqrt{w(e^y)}\dd{y}, \ \forall x\in\mathbb{R}.
	\end{equation*}
	As a result, for any $f\in L^2(\mathbb{R}_+)$,
	\begin{equation*}
		U(S_1-E_1)f(x)=\int_{\mathbb{R}}\sqrt{2^{-d}e^{-(d-1)x}w(e^x)}\frac{\big(Uf\big)(y)}{\cosh[d](\frac{x-y}{2})}\sqrt{2^{-d}e^{-(d-1)y}w(e^y)}\dd{y}, \ \forall x\in\mathbb{R}.
	\end{equation*}
	Then, $S_1-E_1=U^*\alpha_1^{1/2}(X)T_d\alpha_1^{1/2}(X)U$, where
	\begin{equation*}
		\alpha_1(x)=2^{-d}e^{-(d-1)x}w(e^x), \ \forall x\in\mathbb{R},
	\end{equation*}
	and $T_d$ is defined in (\ref{l"}). Notice that $T_d=\mathcal{F}\beta^2(X)\mathcal{F}^*$, where $\beta$ is defined in (\ref{L17}). Therefore,
	\begin{equation*}
		\begin{split}
			S_1-E_1 & =U^*\alpha_1^{1/2}(X)T_d\alpha_1^{1/2}(X)U\\
			& = U^*\alpha_1^{1/2}(X)\mathcal{F}\beta^2(X)\mathcal{F}^*\alpha_1^{1/2}(X)U\\
			& \simeq \beta(X)\mathcal{F}^*\alpha_1(X)\mathcal{F}\beta(X),
		\end{split}
	\end{equation*}
	where the last equivalence is obtained by Lemma \ref{lem} and the fact that $U$ is unitary. Therefore, if $\Psi_1:=\beta(X)\alpha_1(\tfrac{1}{2\pi}D)\beta(X)$, where $\alpha_1$ and $\beta$ are given by (\ref{L17}), then $S_1-E_1$ is unitarily equivalent (modulo kernels) to $\Psi_1$.
\end{proof}

\section{Weyl-type spectral asymptotics}\label{WA}

In this section we derive Weyl-type spectral asymptotics for the operators $\tilde{H}_{\pm1}$, that were defined in \S\ref{PDO}, and for the model operator $\tilde{H}$. The latter is obtained by using the asymptotic orthogonality of $\tilde{H}_1$ and $\tilde{H}_{-1}$; see Lemma \ref{lem4'}.

\begin{lem}\label{lem5'}
	The eigenvalue asymptotics for the operator $\tilde{H}_1$, that was obtained in \S\ref{PDO}, are given by
	\begin{equation}\label{20'}
		\lambda_n^{\pm}(\tilde{H}_1)=C_1^\pm n^{-\gamma}+o(n^{-\gamma}), \ n\to+\infty,
	\end{equation}
	where the constants $C^\pm_1$ are given by a formula similar to (\ref{eq2_12}):
	\begin{equation}\label{19'}
		C^\pm_1=\frac{1}{2^d(d-1)!}\left(b_{1}\right)_{\pm}\left[ \int_{\mathbb{R}} \check{\phi}_d^{\frac{1}{\gamma}}(x) \dd{x} \right]^{\gamma},
	\end{equation}
	where the function $\phi_d$ is defined in (\ref{eq2}) and $\left(b_{1}\right)_{\pm}=\max\{\pm b_1,0\}$. Similar asymptotics are obtained for $\tilde{H}_{-1}$ by substituting $b_1$ with $b_{-1}$ and thus, obtaining the constants $C_{-1}^\pm$.
\end{lem}

\begin{proof}
	In Lemma \ref{lem2'} we proved that $\tilde{H}_1$ is unitarily equivalent (modulo kernels) to an operator $S_1$, so that its spectral asymptotics can be retrieved from those of $S_1$. Moreover, $S_1=(S_1-E_1)+E_1$, where $E_1$ is also described in Lemma \ref{lem2'}. In order to obtain the spectral asymptotics of $S_1$, we aim to use Lemma \ref{lem2_1}. In Lemma \ref{lem2'}, it is proved that $S_1-E_1$ is unitarily equivalent (modulo kernels) to the pseudo-differential operator $\Psi_1=\beta(X)\alpha_1(\tfrac{1}{2\pi}D)\beta(X)$, where $\alpha$ and $\beta$ are given by (\ref{L17}). Then
	\begin{equation*}
		\alpha_1(\tfrac{x}{2\pi})=\left\{
		\begin{array}{lcl}
			\frac{b_1(2\pi)^\gamma}{2^d(d-1)!}|x|^{-\gamma}(1+o(1)) & , & \text{when } x\to-\infty\\ \ \\
			0 &, & \text{when } x\to+\infty
		\end{array}. \right.
	\end{equation*}
	Moreover, $\beta^2$ belongs to the Schwartz class $\mathcal{S}(\mathbb{R})$. Indeed, by differentiating, we can see that $\frac{1}{\cosh[d](\frac{\cdot}{2})}\in\mathcal{S}(\mathbb{R})$ and consequently, $\beta^2\in\mathcal{S}(\mathbb{R})$, too. Therefore,
	$$|\beta(x)|=O\left( \left< x \right>^{-s} \right), \ x\to+\infty,$$
	for every $s>0$. Thus, all the conditions of Lemma \ref{lem2_2} are satisfied and therefore, the eigenvalues of $\Psi_1$, $\lambda_n^\pm(\Psi_1)$, follow the asymptotics below:
	\begin{equation*}
		\lambda_n^{\pm}(\Psi_1)=C_1^\pm n^{-\gamma}+o(n^{-\gamma}), \ n\to+\infty,
	\end{equation*}
	where the constants $C^\pm_1$ are described by (\ref{19'}). Finally, in order to apply Lemma \ref{lem2_1}, it remains to prove that $s_n(E_1)=o(n^{-\gamma})$, for $n\to+\infty$. For notice that, according to Lemma \ref{lem2'}, $E_1\in\cap_{p>0}\mathbf{S}_p$. Thus, the singular values of $E_1$ decay faster than any polynomial. As a result,  Lemma \ref{lem2_1} yields that the eigenvalue asymptotics of $\tilde{H}_1$ are given by (\ref{20'}).
\end{proof}

\begin{lem}\label{lem4'}
	Let $\tilde{H}_1$ and $\tilde{H}_{-1}$ be the operators that were defined in \S\ref{PDO}. Then $\tilde{H}_{-1}\tilde{H}_1$ and $\tilde{H}_{1}\tilde{H}_{-1}$ belong to $\mathbf{S}_p$, for any $p>0$. Therefore, $\tilde{H}_{1}$ and $\tilde{H}_{-1}$ are asymptotically orthogonal.
\end{lem}

\begin{proof}
	First assume that both $b_{-1}$ and $b_1$ are equal to 1, otherwise work with $\frac{1}{b_{-1}}\tilde{H}_{-1}$ or $\frac{1}{b_1}\tilde{H}_1$. In the proof of Lemma \ref{lem2'} we saw that $\tilde{H}_{\pm1}=L^*_{\pm1}L_{\pm1}$. Recall that $L_1$ and $L_1^*$ are given by (\ref{14'}) and (\ref{15'}), respectively, while $L_{-1}$ and $L_{-1}^*$ are defined in (\ref{14''}) and (\ref{15''}). Then
	\begin{equation*}
		\tilde{H}_{-1}\tilde{H}_1=L^*_{-1}L_{-1}L^*_{1}L_{1}.
	\end{equation*}
	Because $L_{-1}^*$ and $L_1$ are bounded, it is enough to prove that $L_{-1}L^*_{1}\in\mathbf{S}_p$, for all $p>0$. To this end, we follow the steps that yielded formula (\ref{16'}) (for $S_1$). Then, for every $f\in L^2(\mathbb{R}_+)$,
	\begin{equation*}
		\begin{split}
			\big(L_{-1}L^*_{1}f\big)(t) & =\int_0^{+\infty}\sqrt{w(t)}\left(\sum_{\mathbf{j}\in\mathbb{N}_0^d}(-1)^{|\mathbf{j}|}e^{-(t+s)|\mathbf{j}|}\right)f(s)\sqrt{w(s)}\dd{s}\\
			& =\int_{0}^{+\infty}\sqrt{w(t)}\frac{f(s)}{\left(1+e^{-(t+s)}\right)^d}\sqrt{w(s)}\dd{s}, \ \forall t\in\mathbb{R}_+.
		\end{split}
	\end{equation*}
	Observe that $(1+e^{-(t+s)})^{-d}\in C^\infty(\mathbb{R})$. Moreover, by the way that the function $\chi_0$ has been defined (see (\ref{eqn0'})), $\sqrt{\chi_0(t)}(1+e^{-(t+s)})^{-d}\sqrt{\chi_0(s)}\in C^\infty_c(\mathbb{R})$. Thus, Lemma \ref{lem3} implies that the integral operator with kernel $\sqrt{\chi_0(t)}(1+e^{-(t+s)})^{-d}\sqrt{\chi_0(s)}$ belongs to any Schatten class $\mathbf{S}_p$. Finally, the same holds true for the operator $L_{-1}L_1^*$, since the function $t^{d-1}|\log t|^{-\gamma}$, is bounded near $0$. In order to prove that $\tilde{H}_{1}\tilde{H}_{-1}\in\bigcap_{p>0}\mathbf{S}_p$, it is enough to notice that $\tilde{H}_{1}\tilde{H}_{-1}=(\tilde{H}_{-1}\tilde{H}_1)^*$.
	
	\noindent Regarding the asymptotic orthogonality, it is enough to notice that, due to Lemma \ref{lem5'}, both $\tilde{H}_{-1}$ and $\tilde{H}_1$ belong to $\mathbf{S}_{\frac{1}{\gamma},\infty}$, and that $\tilde{H}_{-1}\tilde{H}_1$ and $\tilde{H}_{1}\tilde{H}_{-1}$ belong to $\bigcap_{p>0}\mathbf{S}_p\subset\mathbf{S}^0_{\frac{1}{2\gamma},\infty}$.
\end{proof}

\begin{lem}\label{prop0}
	The eigenvalues of the model operator $\tilde{H}$ obey the asymptotic formula below:
	\begin{equation}\label{eq_9}
		\lambda_n^\pm(\tilde{H})=C^{\pm}n^{-\gamma}+o(n^{-\gamma}),
	\end{equation}
	where the constants $C^{\pm}$ are defined in (\ref{eq2_12}).
\end{lem}

\begin{proof}
	According to Lemma \ref{lem5'}, the eigenvalue asymptotics of $\tilde{H}_{1}$ are described by (\ref{20'}) and those of $\tilde{H}_{-1}$ by a similar formula (with constants $C_{-1}^\pm$). But, according to Lemma \ref{lem4'}, $\tilde{H}_{-1}$ and $\tilde{H}_1$ are asymptotically orthogonal. Then, since $\tilde{H}=\tilde{H}_{1}+\tilde{H}_{-1}$, Lemma \ref{lem1} yields that
	\begin{equation*}
		\lambda_n^\pm(\tilde{H})=\big( (C_1)_\pm^{\frac{1}{\gamma}}+(C_{-1})_\pm^{\frac{1}{\gamma}} \big)^\gamma n^{-\gamma}+o(n^{-\gamma}), \ n\to+\infty,
	\end{equation*}
	which gives (\ref{eq_9}).
\end{proof}

\section{Reduction to one-variable weighted Hankel operators}\label{RWHO}

In this section we demonstrate the reduction of multi-variable Hankel matrices to one-variable weighted Hankel operators. This will prove to be a useful tool for the derivation of spectral estimates for the error terms. Define
\begin{equation*}
	W_d(j):=|\{\mathbf{k}\in\mathbb{N}_0^d: \ |\mathbf{k}|=j\}|= {j+d-1 \choose d-1}, \ \forall j\in\mathbb{N}_0;
\end{equation*}
where the last equality can be checked by induction in $d$. In the sequel, consider the linear bounded operator $J: \ell^2(\mathbb{N}_0^d) \to \ell^2(\mathbb{N}_0)$, given by
\begin{equation*}
	\big(J\mathbf{x}\big)(i)=\big(W_d(i)\big)^{-\frac{1}{2}}\sum_{\{\mathbf{k}\in\mathbb{N}_0^d: \ |\mathbf{k}|=i\}}\mathbf{x}(\mathbf{k}), \ \forall i\in\mathbb{N}_0, \ \forall \mathbf{x}\in\ell^2(\mathbb{N}_0^d).
\end{equation*}
Besides, it is not difficult to check that the adjoint of $J$ is given by
\begin{equation*}
	\big( J^*x \big)(\mathbf{i})=\big(W_d(|\mathbf{i}|)\big)^{-\frac{1}{2}}x(|\mathbf{i}|), \ \forall \mathbf{i}\in\mathbb{N}_0^d, \ \forall x\in\ell^2(\mathbb{N}_0).
\end{equation*}
In addition, $J^*$ is an isometry. Indeed, for any $x\in\ell^2(\mathbb{N}_0)$,
\begin{equation*}
	\left\| J^*x \right\|^2 = \sum_{\mathbf{i}\in\mathbb{N}_0^d}\big(W_d(|\mathbf{i}|)\big)^{-1}\big|x(|\mathbf{i}|)\big|^2 = \sum_{i\in\mathbb{N}_0} |x(i)|^2=\|x\|^2.
\end{equation*}
This indicates that $J$ is a partial isometry. Furthermore, for an arbitrary Hankel operator $H_\mathbf{a}:\ell^2(\mathbb{N}_0^d)\to\ell^2(\mathbb{N}_0^d)$, with parameter sequence $\mathbf{a}(\mathbf{j})=a(|\mathbf{j}|)$, $\forall \mathbf{j}\in\mathbb{N}_0^d$ the following relation holds true:
\begin{equation}\label{00}
	(H_\mathbf{a}\mathbf{x},\mathbf{y})=(J^*\Gamma J\mathbf{x},\mathbf{y}), \ \forall \mathbf{x},\mathbf{y}\in\ell^2(\mathbb{N}_0^d),
\end{equation}
where $\Gamma:\ell^2(\mathbb{N}_0)\to \ell^2(\mathbb{N}_0)$ is the weighted Hankel operator defined by
\begin{equation}\label{eq14}
	\left( \Gamma x \right)(i) = \sum_{j\in\mathbb{N}_0}\sqrt{W_d(i)}a(i+j)\sqrt{W_d(j)}x(j), \ \forall i\in\mathbb{N}_0, \ \forall x\in\ell^2(\mathbb{N}_0).
\end{equation}
Indeed, it is enough to observe that, for any $\mathbf{x}$ and $\mathbf{y}$ in $\ell^2(\mathbb{N}_0^d)$,
\begin{equation*}
	\begin{split}
		(H_\mathbf{a}\mathbf{x},\mathbf{y}) & = \sum_{\mathbf{i},\mathbf{j}\in\mathbb{N}_0^d}\mathbf{a}(\mathbf{i}+\mathbf{j})\mathbf{x}(\mathbf{j})\overline{\mathbf{y}(\mathbf{i})}\\
		& = \sum_{\mathbf{i},\mathbf{j}\in\mathbb{N}_0^d}a(|\mathbf{i}+\mathbf{j}|)\mathbf{x}(\mathbf{j})\overline{\mathbf{y}(\mathbf{i})}\\
		& = \sum_{i,j\in\mathbb{N}_0}a(i+j)\sum_{\{\mathbf{k}\in\mathbb{N}_0^d:|\mathbf{k}|=j\}}\mathbf{x}(\mathbf{k}) \overline{\sum_{\{\mathbf{k}\in\mathbb{N}_0^d:|\mathbf{k}|=i\}}\mathbf{y}(\mathbf{k})}\\
		& = \left(\Gamma J\mathbf{x},J\mathbf{y}\right).
	\end{split}
\end{equation*}
This discussion leads to the following lemma.

\begin{lem}\label{label2}
	Let $a=\{a(n)\}_{n\in\mathbb{N}_0}$ and $H_\mathbf{a}:\ell^2(\mathbb{N}^d_0)\to\ell^2(\mathbb{N}^d_0)$ be a Hankel operator with parameter sequence $\mathbf{a}(\mathbf{j})=a(\left|\mathbf{j}\right|)$, for all $\mathbf{j}\in\mathbb{N}^d_0$. Then $H_\mathbf{a}$ belongs to any of the ideals $\mathbf{S}_p$, $\mathbf{S}_{p,q}$, $\mathbf{S}^{0}_{p,\infty}$, where $p>0$ and $q\in(0,+\infty]$, if and only if $\Gamma_{\Phi\phi}^{\frac{d-1}{2},\frac{d-1}{2}}$ does; where $(\Phi\phi)(n)=a(n)$, for all $n\in\mathbb{N}_0$, and $\Gamma_{\Phi\phi}^{\frac{d-1}{2},\frac{d-1}{2}}$ the weighted Hankel operator that is described by the matrix $[(i+1)^{\frac{d-1}{2}}(\Phi\phi)(i+j)(j+1)^{\frac{d-1}{2}}]_{i,j\geq0}$.
\end{lem}

\begin{proof}
	In (\ref{00}), we showed that $H_\mathbf{a}$ and $\Gamma$ have unitarily equivalent non-zero parts. Thus, the operators $H_\mathbf{a}$ and $\Gamma$ have identical non-zero spectra. Besides, it is easily verified that $\Gamma=\mathcal{D}\Gamma_{\Phi\phi}^{\frac{d-1}{2},\frac{d-1}{2}}\mathcal{D}$, where $\mathcal{D}=[\mathcal{D}(j)]_{j\in\mathbb{N}_0}$ is a diagonal matrix defined by $\mathcal{D}(j)=\left(\frac{W_d(j)}{(j+1)^{d-1}}\right)^{\frac{1}{2}}$, and $(\Phi\phi)(j)=a(j)$, $\forall j\in\mathbb{N}_0$. Notice that $\mathcal{D}$ is an invertible bounded operator on $\ell^2(\mathbb{N}_0)$. The boundedness of $\mathcal{D}$ can be checked by noticing that $W_d(j)\sim\frac{j^{d-1}}{(d-1)!}$, when $j\to+\infty$. Therefore, since the classes $\mathbf{S}_p$, $\mathbf{S}_{p,q}$ and $\mathbf{S}^{0}_{p,\infty}$ are ideals of compact operators, and $\mathcal{D}$ is an invertible bounded operator, $\Gamma$ belongs to any of these ideals if and only if $\Gamma_{\Phi\phi}^{\frac{d-1}{2},\frac{d-1}{2}}$ does. Thus, the observation that the non-zero spectra of $H_\mathbf{a}$ and $\Gamma$ are identical gives the result.
\end{proof}

\section{Schatten class inclusions of the error terms}\label{SCI}

In this section we present the spectral estimates of Hankel matrices $H_{\mathbf{a}}$ with parameter sequence $\mathbf{a}(\mathbf{j})=a(\left|\mathbf{j}\right|)$, where $a(j)$ decays faster than $j^{-d}\left|\log j\right|^{-\gamma}$ at infinity, for some positive $\gamma$. These estimates will eventually yield the spectral estimates of the error terms $g_1(j)$ and $(-1)^{j}g_{-1}(j)$, that are defined in Theorem \ref{thm0}.

Let $\mathbf{v}=\{\mathbf{v}(\mathbf{j})\}_{\mathbf{j}\in\mathbb{N}^d_0}$ be a sequence that attains positive values. For any $p\in(0,+\infty)$, we define the spaces $\ell_{\mathbf{v}}^{p}(\mathbb{N}_0^d)$ and $\ell_{\mathbf{v}}^{p,\infty}(\mathbb{N}_0^d)$ as follows:
\begin{equation*}
	\begin{array}{ccc}
		\mathbf{x}\in\ell_{\mathbf{v}}^{p}(\mathbb{N}_0^d) & \Leftrightarrow & \|\mathbf{x}\|_{\ell_{\mathbf{v}}^{p}}:=\displaystyle\left(\sum_{\mathbf{j}\in\mathbb{N}_{0}^{d}}|\mathbf{x}(\mathbf{j})|^p\mathbf{v}(\mathbf{j})\right)^{\frac{1}{p}}<+\infty, \ p\in(0,+\infty),\\ \ \\
		\mathbf{x}\in\ell_{\mathbf{v}}^{p,\infty}(\mathbb{N}_0^d) & \Leftrightarrow & \|\mathbf{x}\|_{\ell_{\mathbf{v}}^{p,\infty}}:= \displaystyle\sup_{\lambda>0}\lambda \left(\sum_{\left\{\mathbf{j}\in\mathbb{N}_{0}^{d}: \ |\mathbf{x}(\mathbf{j})|>\lambda\right\}}\mathbf{v}(\mathbf{j})\right)^{\frac{1}{p}}.
	\end{array}
\end{equation*}
For $p=+\infty$, the space $\ell_{\mathbf{v}}^{\infty}(\mathbb{N}_0^d)$ is identified with the usual $\ell^{\infty}(\mathbb{N}_0^d)$. The case of $\gamma\in(0,\frac{1}{2})$ will be addressed by using the following interpolation lemma.

\begin{lem}\label{lem_0}
	Let $H_{\mathbf{a}}$ be a Hankel matrix with parameter sequence $\mathbf{a}$ and, $\mathbf{v}=\{\mathbf{v}(\mathbf{j})\}_{\mathbf{j}\in\mathbb{N}^d_0}$, with $\mathbf{v}(\mathbf{j})=\left(|\mathbf{j}|+1\right)^{-d}, \ \forall \mathbf{j}\in\mathbb{N}_0^d$. Then, for any $p\in[2,+\infty)$, there exists a positive constant $M_p$ such that
	\begin{equation}\label{label1}
		\|H_\mathbf{a}\|_{\mathbf{S}_{p,\infty}}\leq M_p \left\|\frac{\mathbf{a}}{\mathbf{v}}\right\|_{\ell_{\mathbf{v}}^{p,\infty}}.
	\end{equation}
\end{lem}

\begin{proof}
	The proof is based on the real interpolation method (cf. \cite[Chapter 3]{bergh2012interpolation}). For observe that
	\begin{equation*}
		\begin{split}
			\|H_\mathbf{a}\|_{\mathbf{S}_2}^{2} & = \sum_{\mathbf{i},\mathbf{j}\in\mathbb{N}_0^d}|\mathbf{a}(\mathbf{i}+\mathbf{j})|^2\\
			& = \sum_{i_1,j_1\geq 0}\sum_{i_2,j_2\geq 0}\dots\sum_{i_d,j_d\geq 0}|\mathbf{a}(i_1+j_1,i_2+j_2,\dots,i_d+j_d)|^2\\
			& = \sum_{j_1,j_2,\dots,j_d\geq 0}(j_1+1)(j_2+1)\dots(j_d+1)|\mathbf{a}(j_1,j_2,\dots,j_d)|^2\\
			& \leq \sum_{\mathbf{j}\in\mathbb{N}_0^d}(|\mathbf{j}|+1)^d|\mathbf{a}(\mathbf{j})|^2\\
			& = \sum_{\mathbf{j}\in\mathbb{N}_0^d}(|\mathbf{j}|+1)^{2d}|\mathbf{a}(\mathbf{j})|^2 (|\mathbf{j}|+1)^{-d},
		\end{split}
	\end{equation*}
	so that $\|H_\mathbf{a}\|_{\mathbf{S}_2}\leq\|\frac{\mathbf{a}}{\mathbf{v}}\|_{\ell_{\mathbf{v}}^2}$. In addition, if $\frac{\mathbf{a}}{\mathbf{v}}\in\ell^{\infty}$, then
	\begin{equation*}
		|\mathbf{a}(\mathbf{j})|\leq\frac{\|\frac{\mathbf{a}}{\mathbf{v}}\|_{\ell^\infty}}{(|\mathbf{j}|+1)^d} \leq \frac{\|\frac{\mathbf{a}}{\mathbf{v}}\|_{\ell^\infty}}{(j_1+1)(j_2+1)\dots(j_d+1)}, \ \forall \mathbf{j}\in\mathbb{N}_0^d.
	\end{equation*}
	Thus,
	\begin{equation*}
		\begin{split}
			\left| (H_\mathbf{a}\mathbf{x},\mathbf{y}) \right| & \leq \sum_{\mathbf{i},\mathbf{j}\in\mathbb{N}_{0}^{d}} \left|\mathbf{a}(\mathbf{i}+\mathbf{j})\right|\left|\mathbf{x}(\mathbf{j})\right|\left|\mathbf{y}(\mathbf{i})\right|\\
			& \leq \left\| \frac{\mathbf{a}}{\mathbf{v}} \right\|_{\infty} \sum_{i_1,\dots,i_d,j_1,\dots,j_d\geq 0}\frac{\left|\mathbf{x}(\mathbf{j})\right|\left|\mathbf{y}(\mathbf{i})\right|}{(i_1+j_1+1)\dots(i_d+j_d+1)}\\
			& \leq \pi^d \left\| \frac{\mathbf{a}}{\mathbf{v}} \right\|_{\infty} \left\|\mathbf{x}\right\| \left\|\mathbf{y}\right\|, \ \forall \mathbf{x},\mathbf{y}\in\ell^2(\mathbb{N}_0^d),
		\end{split}
	\end{equation*}
	where the last line is derived from the boundedness of the tensor product of $d$ Hilbert matrices. Therefore, we have shown that there are constants $M_2=1$ and $M_\infty=\pi^d$ such that
	\begin{equation*}
		\|H_\mathbf{a}\|_{\mathbf{S}_2} \leq M_2 \left\| \frac{\mathbf{a}}{\mathbf{v}} \right\|_{\ell_{\mathbf{v}}^2} \ \ \text{and } \ \|H_\mathbf{a}\|\leq M_\infty \left\| \frac{\mathbf{a}}{\mathbf{v}} \right\|_{\ell^\infty},
	\end{equation*}
	and the real interpolation implies that, for any $p\in(2,+\infty)$, there exists a positive constant $M_p$ such that (\ref{label1}) holds true.
\end{proof}

\begin{lem}\label{thm_1}
	Let $\gamma>0$, $M(\gamma)$ be defined in (\ref{8}), and $\{a(j)\}_{j\in\mathbb{N}_{0}}$ be a real valued sequence that satisfies
	\begin{equation}\label{eq_3}
		a^{(m)}(j)=O\big(j^{-d-m} (\log j)^{-\gamma} \big), \ j\to+\infty,
	\end{equation}
	for every $m=0,1,\dots,M(\gamma)$. Then the Hankel operator $H_\mathbf{a}$, with parameter sequence $\mathbf{a}(\mathbf{j})=a(|\mathbf{j}|)$, for all $\mathbf{j}\in\mathbb{N}_0^d$, is compact and its singular values satisfy the following estimate
	\begin{equation}\label{label}
		s_n(H_\mathbf{a})=O(n^{-\gamma}), \ n\to+\infty.
	\end{equation}
	In addition, there exists a positive constant $C_{\gamma}=C(\gamma)$ such that
	\begin{equation}\label{eq_4}
		\|H_\mathbf{a}\|_{\mathbf{S}_{p,\infty}}\leq C_{\gamma} \sum_{m=0}^{M(\gamma)}\sup_{j\geq 0}(j+1)^{d+m}\big( \log(j+2) \big)^{\gamma}|a^{(m)}(j)|,
	\end{equation}
	where $p=\frac{1}{\gamma}$.
\end{lem}

\begin{proof}
We split the proof in steps. In the first step, we prove the result when $\gamma\in(0,\frac{1}{2})$. In the second step, we treat the case of $\gamma\geq\frac{1}{2}$.\\ \ \\
\noindent\underline{\textbf{Step 1:}} Let $\gamma\in(0,\frac{1}{2})$. Then $p=\frac{1}{\gamma}\in(2,+\infty)$ and thus, in order to prove that $H_\mathbf{a}\in\mathbf{S}_{p,\infty}$, it is enough to apply Lemma \ref{lem_0}. To this end, it only needs to show that if $a$ satisfies (\ref{eq_3}), then $\frac{\mathbf{a}}{\mathbf{v}}\in\ell_{\mathbf{v}}^{p,\infty}$, for every $p\in(2,+\infty)$, where $\mathbf{v}$ is defined in Lemma \ref{lem_0}. For $\lambda>0$,
\begin{equation*}
	\begin{split}
		\left\{ \mathbf{j}\in\mathbb{N}_0^d: \ \frac{|\mathbf{a}(\mathbf{j})|}{\mathbf{v}(\mathbf{j})}>\lambda \right\} & = \left\{ \mathbf{j}\in\mathbb{N}_0^d: \ (|\mathbf{j}|+1)^d |a(|\mathbf{j}|)|>\lambda \right\}\\
		& \subset \left\{ \mathbf{j}\in\mathbb{N}_0^d: \ \log(|\mathbf{j}|+2)<\left(\frac{A_0}{\lambda}\right)^p \right\},
	\end{split}
\end{equation*}
where $A_0:=\sup_{j\geq 0}(j+1)^d\big( \log(j+2) \big)^{\gamma}|a(j)|$. Therefore,
	\begin{equation*}
		\begin{split}
			\sum_{\left\{\mathbf{j}\in\mathbb{N}_{0}^{d}: \ \frac{|\mathbf{a}(\mathbf{j})|}{\mathbf{v}(\mathbf{j})}>\lambda \right\}}\frac{1}{(|\mathbf{j}|+1)^d} & \lesssim \sum_{\left\{\mathbf{j}\in\mathbb{N}_{0}^{d}: \ \log(|\mathbf{j}|+2)< \left(\frac{A_0}{\lambda}\right)^{p} \right\}}\frac{1}{(|\mathbf{j}|+2)^d}\\
			& = \sum_{\left\{j\in\mathbb{N}_{0}: \ \log(j+2)< \left(\frac{A_0}{\lambda}\right)^{p} \right\}}\frac{W_d(j)}{(j+2)^d}\\
			& \lesssim \sum_{\left\{ j\in\mathbb{N}_0: \ \log(j+2)<\left( \frac{A_0}{\lambda} \right)^p \right\}} \frac{(j+2)^{d-1}}{(j+2)^d}\\
			& \lesssim \int_{\left\{ \log(x+2)<\left( \frac{A_0}{\lambda} \right)^p \right\}} \frac{1}{x+2}\dd{x}\\
			& \lesssim\left(\frac{A_0}{\lambda}\right)^p.
		\end{split}
	\end{equation*}
	Thus, there exists a positive constant $C$ such that
	$$\lambda^p \sum_{\left\{\mathbf{j}\in\mathbb{N}_{0}^{d}: \  \frac{|\mathbf{a}(\mathbf{j})|}{\mathbf{v}(\mathbf{j})}>\lambda \right\}}\frac{1}{(|\mathbf{j}|+1)^d}\leq (CA_0)^p,$$
	which implies, by taking supremum over positive $\lambda$'s, that $\left\| \frac{\mathbf{a}}{\mathbf{v}} \right\|_{\ell_{\mathbf{v}}^{p,\infty}}\leq CA_0$. From the last relation and Lemma \ref{lem_0}, we conclude that
	$$\|H_\mathbf{a}\|_{\mathbf{S}_{p,\infty}}\leq M_p\left\| \frac{\mathbf{a}}{\mathbf{v}} \right\|_{\ell_{\mathbf{v}}^{p,\infty}}\leq M_p CA_0,$$
	so that relation (\ref{eq_4}) comes true, by setting $C_{\gamma}=M_pC$.\\ \ \\	
	\noindent\underline{\textbf{Step 2:}} Assume that $\gamma\geq\frac{1}{2}$ and let $\phi$ be given by
	\begin{equation*}
		\phi(z)=\sum_{j\in\mathbb{N}_0}a(j)z^j, \ \forall z\in\overline{\mathbb{D}}.
	\end{equation*}
	According to Lemma \ref{label2}, $H_\mathbf{a}$ and $\Gamma_{\Phi\phi}^{\frac{d-1}{2},\frac{d-1}{2}}$ satisfy the same Schatten class inclusions, where $(\Phi\phi)(j)=a(j)$. Therefore, in order to derive (\ref{label}), Lemma \ref{lem0} suggests that it is enough to show that $\bigoplus_{n\geq 0}2^{n(d-1)}\phi*V_n\in L^{p,\infty}(\mathcal{M},\mu)$, or, in other words, that
	\begin{equation}\label{eq11}
		\sup_{s>0} s^p \sum_{n\in\mathbb{N}_0} 2^{n}\big| \{ t\in [-\pi,\pi): |2^{n(d-1)}(\phi*V_n)(e^{it})|>s \} \big|<+\infty.
	\end{equation}
	For every non-negative integer $n$ and any positive number $s$, set
	\begin{equation*}
		E_n(s):=\{ t\in[-\pi,\pi): |2^{n(d-1)}(\phi*V_n)(e^{it})|>s \}.
	\end{equation*}
	The goal is to find an estimate for $|E_n(s)|$ which proves the finiteness of (\ref{eq11}). First of all, notice that $E_n(s)=\varnothing$, for every $s\geq\|2^{n(d-1)}\phi*V_n\|_{\infty}$. An application of (\ref{eq3}) gives that $E_n(s)=\varnothing$, for every $s\geq 2^{n(d-1)} \sum_{j=2^{n-1}}^{2^{n+1}}|a(j)|$. Let
	$$A_m:=\sup_{j\geq 0}\big|a^{(m)}\big|(j+1)^{d+m}\big( \log(j+2) \big)^{\gamma}, \ \forall m=0,1,\dots,M(\gamma).$$
	Therefore, condition (\ref{eq_3}) implies that $E_n(s)=\varnothing$ when
	\begin{equation*}
		s\geq2^{n(d-1)} A_0 \sum_{j=2^{n-1}}^{2^{n+1}} (j+1)^{-d} \big( \log(j+2) \big)^{-\gamma}.
	\end{equation*}
	Besides, for every $n\geq3$,
	\begin{equation*}
		\begin{split}
			\sum_{j=2^{n-1}}^{2^{n+1}} (j+1)^{-d} \big( \log(j+2) \big)^{-\gamma} & \leq \int_{2^{n-1}-1}^{2^{n+1}} (t+1)^{-d} \big( \log(t+2) \big)^{-\gamma}\dd{t}\\
			& \lesssim \int_{n-1}^{n+1} 2^{-s(d-1)} s^{-\gamma}\dd{s}\ (\text{change of variable} \ s=\log_2t)\\
			& \lesssim 2^{-n(d-1)}n^{-\gamma},
		\end{split}
	\end{equation*}
	so that in general,
	$$\sum_{j=2^{n-1}}^{2^{n+1}} (j+1)^{-d} \big( \log(j+2) \big)^{-\gamma} \leq C 2^{-n(d-1)} \left< n \right>^{-\gamma}, \ \forall n\geq 0,$$
	for some positive constant $C$; without loss of generality, we may assume that $C=C_q$, where $C_q$ appears in (\ref{eq4}). Therefore, $E_n(s)=\varnothing$, for every $n\geq 0$ such that $\left< n \right>\geq N(s)$, where $N(s):=(\frac{C_qA_0}{s})^p$, $\forall s>0$. Besides, by following exactly the same steps, it can be shown that
	\begin{equation*}
		\sum_{j=2^{n-1}-M(\gamma)}^{2^{n+1}}(j+1)^m|a^{(m)}(j)|\lesssim A_m 2^{-n(d-1)}\left< n \right>^{-\gamma}, \ \forall m=1,2,\dots,M(\gamma).
	\end{equation*}
	Thus, Lemma \ref{lem_2} gives that, for every $q>\frac{1}{M(\gamma)}$ and $n\in\mathbb{N}_0$ such that $M(\gamma)\leq 2^{n-1}$,
	\begin{equation*}
		2^n\|\phi*V_n\|_q^q\lesssim C_q A^q 2^{-n(d-1)q}\left< n \right>^{-\gamma q},
	\end{equation*}
	where $A:=\sum_{m=0}^{M(\gamma)}A_m$.
	Now notice that, for any positive $q$,
	\begin{equation*}
		\begin{split}
			s^q|E_n(s)| & = \int\limits_{\left\{t\in[-\pi,\pi): \ 2^{n(d-1)}|(\phi*V_n)(e^{it})|>s\right\}} s^q\dd{t}\\
			& \leq 2^{n(d-1)q} \|\phi*V_n\|_q^q, \ \forall n\in\mathbb{N}_0.
		\end{split}
	\end{equation*}
	Putting all these together results that, for every $q\in (\frac{1}{M(\gamma)},p)$ and $s>0$,
	\begin{equation*}
		\begin{split}
			s^p\sum_{n\in\mathbb{N}_0}2^n |E_n(s)| & =s^{p-q} \left( s^q \sum_{\left<n\right>\leq N(s)}2^n |E_n(s)| \right)\\
			& \leq s^{p-q} \sum_{\left<n\right>\leq N(s)}2^n 2^{n(d-1)q}\|\phi*V_n\|_q^q\\
			& \lesssim s^{p-q} C_q A^q \sum_{\left<n\right>\leq N(s)} 2^{n(d-1)q} 2^{-n(d-1)q} \left< n \right>^{-\gamma q}\\
			& \lesssim C_qA^q s^{p-q} N^{1-\gamma q}(s), \ (\text{since} \ \gamma q=\frac{q}{p} \ \text{and} \ q<p).
		\end{split}
	\end{equation*}
	Finally, notice that $s^{p-q}N^{1-\gamma q}(s)= s^{p-q} (C_qA_0)^{p-q}s^{-(p-q)}=(C_qA_0)^{p-q},$ so there is a positive constant $K$, independent of $s$, such that
	$$s^p\sum_{n\in\mathbb{N}_0}2^n |E_n(s)|\leq K^p A^p, \ \forall s>0,$$
	and this proves the desired result. Finally, Lemma \ref{lem0} suggests that there is a positive constant $K_{\gamma}$ such that
	$$\|H_\mathbf{a}\|_{\mathbf{S}_{p,\infty}}=\|\Gamma\|_{\mathbf{S}_{p,\infty}}\leq K_{\gamma} \|\phi\|_{\mathcal{B}_{p,\infty}^{\frac{1}{p}+d-1}}\leq K_{\gamma} K A,$$
	where $\Gamma$ is given by (\ref{eq14}). This gives relation (\ref{eq_4}), with $C_{\gamma}=K_{\gamma}K$.
\end{proof}

\begin{lem}\label{thm2_7}
	Let $\gamma>0$ and $\{a(j)\}_{j\in\mathbb{N}_{0}}$ be a real valued sequence such that
	\begin{equation}\label{eq12}
		a^{(m)}(j)=o\big(j^{-d-m} (\log j)^{-\gamma} \big), \ j\to+\infty,
	\end{equation}
	for every $m=0,1,\dots,M(\gamma)$. Then the Hankel operator $H_\mathbf{a}$, with parameter sequence $\mathbf{a}(\mathbf{j})=a(|\mathbf{j}|)$, for all $\mathbf{j}\in\mathbb{N}_0^d$, is compact and its singular values satisfy the following estimate
	$$s_n(H_\mathbf{a})=o(n^{-\gamma}), \ n\to+\infty.$$
\end{lem}

\begin{proof}
	The goal is to show that $H_\mathbf{a}\in\mathbf{S}_{p,\infty}^{0}$, for $p=\frac{1}{\gamma}$. The ideal $\mathbf{S}_{p,\infty}^{0}$ is the $\left\|\cdot\right\|_{\mathbf{S}_{p,\infty}}$-closure of finite rank operators. So, it is enough to approximate $H_\mathbf{a}$ by finite rank operators in the $\left\|\cdot\right\|_{\mathbf{S}_{p,\infty}}$ quasi-norm. For consider the cut-off function
	\begin{equation*}
		\chi_0(t)=\begin{cases}
			1, & t\in[0,1]\\
			0, & t\geq2,
		\end{cases}
	\end{equation*}
such that $\chi_0\in C^{\infty}(\mathbb{R}_+)$ and $0\leq\chi_0\leq1$. In addition, for every $N\in\mathbb{N}$, define the sequences
\begin{equation*}
	q_N(j)=a(j)\chi_0(\tfrac{j}{N}) \ \text{ and } \ h_N(j)=a(j)-q_N(j), \ \forall j\in\mathbb{N}_0.
\end{equation*}
Let $H_{\mathbf{q}_N}$ and $H_{\mathbf{h}_N}$ be the Hankel operators, with parameter sequences $\mathbf{q}_N(\mathbf{j})=q_N(|\mathbf{j}|)$ and $\mathbf{h}_N(\mathbf{j})=h_N(|\mathbf{j}|)$, $\forall \mathbf{j}\in\mathbb{N}_0^d$, respectively. In other words, $H_{\mathbf{h}_N}=H_\mathbf{a}-H_{\mathbf{q}_N}$. Then, by using the Leibniz rule,
\begin{equation*}
	\big(h_N\big)^{(m)}(j)=\sum_{n=0}^{m}{m \choose n}a^{(m-n)}(j+n)\big(1-\chi_0\big)^{(n)}(\tfrac{j}{N}), \ \forall j\in\mathbb{N}_0.
\end{equation*}
Therefore, for every $j\geq2$,
\begin{equation}\label{1}
	\begin{split}
		\left|\big(h_N\big)^{(m)}(j)j^{d+m}(\log j)^{\gamma}\right| & \leq \sum_{n=0}^{m}{m \choose n}\left|a^{(m-n)}(j+n)\right|j^{d+m-n}(\log j)^\gamma j^n\big(1-\chi_0\big)^{(n)}(\tfrac{j}{N})\\
		& \leq \sum_{n=0}^{m}{m \choose n}\left|a^{(m-n)}(j+n)\right|(j+n)^{d+m-n}(\log (j+n))^\gamma j^n\big(1-\chi_0\big)^{(n)}(\tfrac{j}{N}).
	\end{split}
\end{equation}
Moreover, observe that, for any $n\in\mathbb{N}$,
\begin{equation*}
	t^n\big(1-\chi_{0}\big)^{(n)}\big(\tfrac{t}{N}\big)=-\big(\tfrac{t}{N}\big)^n\chi_{0}^{(n)}\big(\tfrac{t}{N}\big), \ \forall t\in\mathbb{R}_+.
\end{equation*}
As a result, by recalling that $\chi_0$ is compactly supported, there exist positive constants $K_n$, $n=1,2,\dots,M(\gamma)$, independent of $N$, such that
	\begin{equation*}
		\sup_{t>0}t^n \big(1-\chi_{0}\big)^{(n)}\big(\tfrac{t}{N}\big)\leq K_n.
	\end{equation*}
	By considering $K:=\max_{1\leq n\leq M(\gamma)}\{K_n\}$, (\ref{1}) gives
	\begin{equation*}
		\left|\big(h_N\big)^{(m)}(j)\right|j^{d+m}(\log j)^{\gamma} \leq K \sum_{n=0}^{m}{m \choose n}\left|a^{(m-n)}(j+n)\right|(j+n)^{d+m-n}(\log (j+n))^\gamma, \ \forall j\geq 2.
	\end{equation*}
	By taking supremum,
	\begin{equation}\label{5}
		\sup_{j>N}\left|\big(h_N\big)^{(m)}(j)\right|j^{d+m}(\log j)^{\gamma} \lesssim \sum_{n=0}^{m}{m \choose n}\sup_{j>N}\left|a^{(m-n)}(j+n)\right|(j+n)^{d+m-n}(\log (j+n))^\gamma.
	\end{equation}
	Under the assumption (\ref{eq12}) for $a$, we see that, for any $N\in\mathbb{N}$, $h_N$ satisfies assumption (\ref{eq_3}) of Lemma \ref{thm_1} and consequently, $H_{\mathbf{h}_N}$ satisfies relation (\ref{eq_4}). Thus, there exists a constant $C_\gamma$ such that
	\begin{equation*}
		\begin{split}
			\left\| H_\mathbf{a}-H_{\mathbf{q}_N} \right\|_{\mathbf{S}_{p,\infty}} & \leq C_\gamma\sum_{m=0}^{M(\gamma)} \sup_{j\in\mathbb{N}_0} \left| \big(h_N\big)^{(m)}(j) \right| (j+1)^{d+m} \big(\log(j+2)\big)^\gamma\\
			& = C_\gamma \sum_{m=0}^{M(\gamma)}\sup_{j>N} \left| \big(h_N\big)^{(m)}(j) \right| (j+1)^{d+m} \big(\log(j+2)\big)^\gamma.
		\end{split}
	\end{equation*}
	Then (\ref{5}) implies that
	\begin{equation}\label{6}
		\left\| H_\mathbf{a}-H_{\mathbf{q}_N} \right\|_{\mathbf{S}_{p,\infty}} \lesssim \sum_{m=0}^{M(\gamma)} \sum_{n=0}^{m}{m \choose n}\sup_{j>N}\left|a^{(m-n)}(j+n)\right|(j+n)^{d+m-n}(\log (j+n))^\gamma.
	\end{equation}
	Notice that assumption (\ref{eq12}) implies that, for any $n\in\mathbb{N}$,
	\begin{equation*}
		\limsup_{j\to+\infty}\left|a^{(m)}(j+n)\right|(j+n)^{d+m}(\log (j+n))^\gamma=0, \ \forall m=0,1,\dots,M(\gamma).
	\end{equation*}
	Thus, letting $N\to+\infty$ in (\ref{6}) results that the right hand side converges to zero and the result is obtained.
\end{proof}

\section{Proof of Theorem \ref{thm0}}

\begin{proof}
Let $\tilde{a}$ be the sequence that is defined in (\ref{eq2_4'}) and generates the model operator $\tilde{H}$. Then $H_\mathbf{a}=\tilde{H}+(H_\mathbf{a}-\tilde{H})$, where $H_\mathbf{a}-\tilde{H}=H_{\mathbf{a}-\tilde{\mathbf{a}}}$ is a Hankel operator with parameter sequence $\big(\mathbf{a}-\tilde{\mathbf{a}}\big)(\mathbf{j})=\big(a-\tilde{a}\big)(|\mathbf{j}|)$, for any $\mathbf{j}\in\mathbb{N}_0^d$. By observing (\ref{eq5}) and Lemma \ref{lem3''}, we see that the sequences $a$ and $\tilde{a}$ present the same asymptotic behaviour, modulo some error terms. Thus,
\begin{equation*}
	\big(a-\tilde{a}\big)(j)=h_1(j)+(-1)^jh_{-1}(j), \ \forall j\geq 2,
\end{equation*}
where $h_{\pm1}(j):=\big(g_{\pm1}-\tilde{g}_{\pm1}\big)(j)$, $\forall j\in\mathbb{N}_0$. Furthermore, notice that relation (\ref{eqc_18'}) in Lemma \ref{lem3''} implies that
\begin{equation*}
	\tilde{g}_{\pm1}^{(m)}=o\big(j^{-d-m}(\log j)^{-\gamma}\big), \ j\to+\infty.
\end{equation*}
The same relation is satisfied by $g_{\pm1}$, as well, by assumption. Therefore,
\begin{equation*}
	h_{\pm1}^{(m)}=o\big(j^{-d-m}(\log j)^{-\gamma}\big), \ j\to+\infty.
\end{equation*}
Consider the Hankel operators $H_{\mathbf{h}_{\pm 1}}:\ell^2(\mathbb{N}_0^d)\to \ell^2(\mathbb{N}_0^d)$, with parameter sequence $\mathbf{h}_{\pm1}(\mathbf{j})=h_{\pm1}(\left|\mathbf{j}\right|)$, for all $\mathbf{j}\in\mathbb{N}_0^d$. Then Lemma \ref{thm2_7} yields that their singular values satisfy the following asymptotic law:
\begin{equation}\label{eq}
	s_n(H_{\mathbf{h}_{\pm 1}})=o(n^{-\gamma}), \ n\to+\infty.
\end{equation}
Let $\mathfrak{h}_{-1}(j):=(-1)^jh_{-1}(j)$, for all $j\in\mathbb{N}_0$, and consider the Hankel operator $H_{\boldsymbol{\mathfrak{h}}_{-1}}:\ell^2(\mathbb{N}_0^d)\to \ell^2(\mathbb{N}_0^d)$, with parameter sequence $\boldsymbol{\mathfrak{h}}_{-1}(\mathbf{j})=\mathfrak{h}_{-1}(\left|\mathbf{j}\right|)$, for all $\mathbf{j}\in\mathbb{N}_0^d$. Then $H_{\boldsymbol{\mathfrak{h}}_{-1}}$ and $H_{\mathbf{h}_{-1}}$ are unitarily equivalent. Indeed, by defining the unitary operator $Q:\ell^2(\mathbb{N}_0^d)\to\ell^2(\mathbb{N}_0^d)$, with
\begin{equation*}
	\big(Q\mathbf{x}\big)(\mathbf{j})=(-1)^{|\mathbf{j}|}\mathbf{x}(\mathbf{j}), \ \forall \mathbf{j}\in\mathbb{N}_0^d, \ \forall \mathbf{x}\in\ell^2(\mathbb{N}_0^d),
\end{equation*}
it is checked easily that $H_{\boldsymbol{\mathfrak{h}}_{-1}}=Q^*H_{\mathbf{h}_{-1}}Q$. Thus, the singular values of $H_{\boldsymbol{\mathfrak{h}}_{-1}}$ satisfy (\ref{eq}). Notice that $H_{\mathbf{a}}-\tilde{H}=H_{\mathbf{h}_1}+H_{\boldsymbol{\mathfrak{h}}_{-1}}$. Therefore, since the space $\mathbf{S}_{p,\infty}^0$ is linear, the singular values of $H_\mathbf{a}-\tilde{H}$ satisfy (\ref{eq}). Finally, recall that the eigenvalue asymptotics of $\tilde{H}$ are given in Lemma \ref{prop0}. Thus by combining with the fact that
\begin{equation*}
	s_n(H_a-\tilde{H})=o(n^{-\gamma}), \ n\to+\infty,
\end{equation*}
Lemma \ref{lem2_1} yields the asymptotic law (\ref{eq1}).
\end{proof}

\appendix

\section{Discussion of Lemma \ref{lem0}}\label{A}

The discussion is provided in order to clarify some subtle points that exist in the proof of \cite[Theorem 6.4.4]{peller2012hankel}. The proof is based on the retract argument, which requires the construction of two specific bounded operators $\mathcal{I}$ and $\mathcal{K}$; for details, see below. This discussion mainly aims to clarify the boundedness of $\mathcal{K}$, which is defined in (\ref{K}).

The claim is equivalent to the fact that the mapping $\phi\mapsto\Gamma_{\Phi\phi}^{\frac{d-1}{2},\frac{d-1}{2}}$, where $\Phi\phi$ is the sequence of the Fourier coefficients of $\phi$, is a bounded linear operator from $\mathcal{B}_{p,q}^{\frac{1}{p}+d-1}$ to $\mathbf{S}_{p,q}$. The boundedness is proved via interpolation arguments.

For the reader's convenience, we recall a real interpolation method (the \textit{K-method}) (cf. \cite[\S 3.1]{bergh2012interpolation}) and the reiteration theorem (cf. \cite[\S 3.5]{bergh2012interpolation}). If $X_0$ and $X_1$ are two quasi-Banach spaces which are continuously embedded into the same Hausdorff topological space, then $X_0$ and $X_1$ are called \textit{compatible}. If $(X_0,X_1)$ is a compatible pair of quasi-Banach spaces, then real interpolation with parameters $\theta\in(0,1)$ and $q\in[1,+\infty]$, or $\theta\in[0,1]$ and $q=+\infty$, results an intermediate quasi-Banach space $X_{\theta,q}:=(X_0,X_1)_{\theta,q}$.

\begin{thm}[Reiteration Theorem]
	Let $X_{\theta_0,q_0}$ and $X_{\theta_1,q_1}$ be two interpolation spaces created by the K-method from the compatible couple of quasi-Banch spaces $(X_0,X_1)$. Then, for any $q\in[1,+\infty]$ and $t\in(0,1)$,
	\begin{equation*}
		(X_{\theta_0,q_0},X_{\theta_1,q_1})_{t,q}=X_{\theta,q} \ \ \text{where } \ \theta=(1-t)\theta_0+t\theta_1.
	\end{equation*}
\end{thm}

An application of real interpolation and the reiteration theorem yields that, for every $p_0\in(0,+\infty)$, $\theta\in(0,1)$ and $q\in(0,+\infty]$,
\begin{equation}\label{Ap}
	\left( \mathbf{S}_{p_0}, \mathbf{S}_{\infty} \right)_{\theta,q}=\mathbf{S}_{p,q}, \ \ \text{where } \ p=\frac{p_0}{1-\theta}.
\end{equation}
For the reader's convenience we give a sketch of proof for (\ref{Ap}) and we refer to \cite{karadzhov1976application} for further details. Relation (\ref{Ap}) holds true for $p_0<q$ (cf. \cite[(11)]{karadzhov1976application}). Then an application of the reiteration theorem yields
\begin{equation}\label{Ap0}
	(\mathbf{S}_{p_0,q_0},\mathbf{S}_{p_1,q_1})_{\theta,q}=\mathbf{S}_{p,q}, \ \ \text{where } \ \frac{1}{p}=\frac{1-\theta}{p_0}+\frac{\theta}{p_1},
\end{equation}
for any $\theta\in(0,1)$ and $q\in(0,+\infty]$ (cf. \cite[(14)]{karadzhov1976application}). Finally, by using (\ref{Ap0}) together with the reiteration theorem, one obtains (\ref{Ap}) for any $q\in(0,+\infty]$.

Now we go back to our main claim; namely, the mapping $\phi\mapsto\Gamma_{\Phi\phi}^{\frac{d-1}{2},\frac{d-1}{2}}$ is a bounded linear operator from $\mathcal{B}_{p,q}^{\frac{1}{p}+d-1}$ to $\mathbf{S}_{p,q}$. Theorem \ref{thm2_8} implies that this mapping represents a bounded linear operator from $B_{p}^{\frac{1}{p}+d-1}$ to $\mathbf{S}_p\subset\mathbf{S}_\infty$, $\forall p\in(0,+\infty)$. Therefore, with due regard to (\ref{Ap}), it is enough to prove that, for every $p_0,p_1\in(0,+\infty)$, $\theta\in(0,1)$ and $q\in(0,+\infty]$,
\begin{equation}\label{A0}
	\left( B_{p_0}^{\frac{1}{p_0}+d-1},B_{p_1}^{\frac{1}{p_1}+d-1} \right)_{\theta,q}=\mathcal{B}_{p,q}^{\frac{1}{p}+d-1}, \ \ \text{where } \ \frac{1}{p}=\frac{1-\theta}{p_0}+\frac{\theta}{p_1}.
\end{equation}
The relation (\ref{A0}) will be proved by using the retract argument \cite[Theorem 6.4.2]{bergh2012interpolation}. Briefly, if $X$ and $Y$ are two quasi-Banach spaces, then $X$ is a \textit{retract} of $Y$ if there are bounded linear mappings $\mathcal{J}:X\to Y$ and $\mathcal{K}:Y\to X$ such that $\mathcal{K}\mathcal{J}$ is the identity map on $X$.

\noindent Notice that for every $p_0,p_1\in(0,+\infty)$, $\theta\in(0,1)$ and $q\in(0,+\infty]$,
\begin{equation}\label{A1}
	\big( L^{p_0}(\mathcal{M},\mu),L^{p_1} (\mathcal{M},\mu)\big)_{\theta,q}=L^{p,q}(\mathcal{M},\mu), \ \ \text{where } \ \frac{1}{p}=\frac{1-\theta}{p_0}+\frac{\theta}{p_1}.
\end{equation}
The relation above can be found in \cite[Theorem 5.3.1]{bergh2012interpolation}. The goal is to construct two mappings $\mathcal{J}:B_{p_0}^{\frac{1}{p_0}+d-1}+B_{p_1}^{\frac{1}{p_1}+d-1}\to L^{p_0}(\mathcal{M},\mu)+L^{p_1}(\mathcal{M},\mu)$ and $\mathcal{K}:L^{p_0}(\mathcal{M},\mu)+L^{p_1}(\mathcal{M},\mu)\to B_{p_0}^{\frac{1}{p_0}+d-1}+B_{p_1}^{\frac{1}{p_1}+d-1}$ such that:
\begin{enumerate}[(i)]
	\item $B_{p_0}^{\frac{1}{p_0}+d-1}$ is retract of $L^{p_0}(\mathcal{M},\mu)$, under the mappings $\mathcal{J}:B_{p_0}^{\frac{1}{p_0}+d-1}\to L^{p_0}(\mathcal{M},\mu)$ and $\mathcal{K}:L^{p_0}(\mathcal{M},\mu)\to B_{p_0}^{\frac{1}{p_0}+d-1}$; and
	\item $B_{p_1}^{\frac{1}{p_1}+d-1}$ is a retract of $L^{p_1} (\mathcal{M},\mu)$ under the mappings $\mathcal{J}:B_{p_1}^{\frac{1}{p_1}+d-1}\to L^{p_1} (\mathcal{M},\mu)$ and $\mathcal{K}:L^{p_1} (\mathcal{M},\mu)\to B_{p_1}^{\frac{1}{p_1}+d-1}$.
\end{enumerate}
For let $Hol(\mathbb{D})$ be the space of the holomorphic functions on $\mathbb{D}$ and define the linear operator
\begin{equation*}
	\mathcal{J}\phi=\bigoplus_{n\in\mathbb{N}_0}2^{n(d-1)}\phi*V_n, \ \forall \phi\in Hol(\mathbb{D}),
\end{equation*}
where the polynomials $V_n$ are defined in (\ref{1'}) and (\ref{2}). Then, by the definition of the Besov space $B_{p}^{\frac{1}{p}+d-1}$, $\mathcal{J}$ is an isometry from $B_{p}^{\frac{1}{p}+d-1}$ to $L^p(\mathcal{M},\mu)$. In addition, consider the polynomials
\begin{equation*}
	\tilde{V}_0(z)=V_0(z)+V_1(z), \ \forall z\in\mathbb{T},
\end{equation*}
and, for every $n\in\mathbb{N}$,
\begin{equation*}
	\tilde{V}_n(z)=V_{n-1}(z)+V_n(z)+V_{n+1}(z), \ \forall z\in\mathbb{T}.
\end{equation*}
Notice that $V_n*\tilde{V}_n=V_n$, for every $n\in\mathbb{N}_0$. Now define the linear operator
\begin{equation}\label{K}
	\mathcal{K}\bigoplus_{n\in\mathbb{N}_0}\phi_n=\sum_{n\in\mathbb{N}_0}2^{-n(d-1)}\phi_n*\tilde{V}_n, \ \forall \bigoplus_{n\in\mathbb{N}_0}\phi_n\in L^p(\mathcal{M},\mu),
\end{equation}
which is bounded from $L^p(\mathcal{M},\mu)$ to $B_{p}^{\frac{1}{p}+d-1}$. To see this, it is enough to check that
\begin{equation}\label{eq2_20}
	\sum_{n\in\mathbb{N}_0} 2^{n[1+p(d-1)]}\big\| \sum_{m\in\mathbb{N}_0}2^{-m(d-1)}\phi_m*\tilde{V}_m*V_n \big\|_p^p<+\infty, \ \forall \bigoplus_{n\in\mathbb{N}_0}\phi_n\in L^p(\mathcal{M},\mu).
\end{equation}
Moreover, notice that $\mathcal{KJ}\phi=\phi$, for all $\phi$.

It has been proved that $\mathcal{J}$ is an isometry from $B_{p}^{\frac{1}{p}+d-1}$ to $L^p(\mathcal{M},\mu)$. Therefore, due to (\ref{A1}), if $\phi$ belongs to $\left( B_{p_0}^{\frac{1}{p_0}+d-1},B_{p_1}^{\frac{1}{p_1}+d-1} \right)_{\theta,q}$, then $\mathcal{J}\phi\in L^{p,q}(\mathcal{M},\mu)$, where $p$ is described in (\ref{A0}). According to the definition of $\mathcal{B}_{p,q}^{\frac{1}{p}+d-1}$ (see (\ref{4})), $\phi\in\mathcal{B}_{p,q}^{\frac{1}{p}+d-1}$ and thus,
\begin{equation}\label{A2}
	\left( B_{p_0}^{\frac{1}{p_0}+d-1},B_{p_1}^{\frac{1}{p_1}+d-1} \right)_{\theta,q} \subset\mathcal{B}_{p,q}^{\frac{1}{p}+d-1}.
\end{equation}
On the other hand, let $\phi\in\mathcal{B}_{p,q}^{\frac{1}{p}+d-1}$ or equivalently, $\mathcal{J}\phi\in L^{p,q}(\mathcal{M},\mu)$. Moreover, $\phi=\mathcal{KJ}\phi$ and $\mathcal{K}$ is bounded from $L^p(\mathcal{M},\mu)$ to $B_{p}^{\frac{1}{p}+d-1}$. Therefore, due to (\ref{A1}), $\phi\in\left( B_{p_0}^{\frac{1}{p_0}+d-1},B_{p_1}^{\frac{1}{p_1}+d-1} \right)_{\theta,q}$. This results
\begin{equation}\label{A3}
	\mathcal{B}_{p,q}^{\frac{1}{p}+d-1}\subset\left( B_{p_0}^{\frac{1}{p_0}+d-1},B_{p_1}^{\frac{1}{p_1}+d-1} \right)_{\theta,q}.
\end{equation}
Therefore, (\ref{A2}) and (\ref{A3}) yield (\ref{A0}).

In order to complete the proof, it only remains to verify the validity of (\ref{eq2_20}). For notice that
\begin{equation*}
	\sum_{m\in\mathbb{N}_0}\phi_m*\tilde{V}_m*V_n=\phi_{n-1}*\tilde{V}_{n-1}*V_n+\phi_n*V_n+\phi_{n+1}*\tilde{V}_{n+1}*V_n, \ \forall n\in\mathbb{N}.
\end{equation*}
Thus, for any $p>0$ and every $n\in\mathbb{N}$,
\begin{multline}\label{eq2_19}
	\big\| \sum_{m\in\mathbb{N}_0}2^{-m(d-1)}\phi_m*\tilde{V}_m*V_n \big\|_p^p\lesssim  \|2^{-(n-1)(d-1)}\phi_{n-1}*\tilde{V}_{n-1}*V_n\|_p^p+\|2^{-n(d-1)}\phi_{n}*V_n\|_p^p \\ +\|2^{-(n+1)(d-1)}\phi_{n+1}*\tilde{V}_{n+1}*V_n\|_p^p.
\end{multline}
Observe that if the convolution with $V_n$ is a bounded operator whose norm does not depend on $n$, then (\ref{eq2_19}) becomes
\begin{equation*}
	\big\| \sum_{m\in\mathbb{N}_0}2^{-m(d-1)}\phi_m*\tilde{V}_m*V_n \big\|_p^p\lesssim \|2^{-(n-1)(d-1)}\phi_{n-1}\|_p^p+\|2^{-n(d-1)}\phi_{n}\|_p^p +\|2^{-(n+1)(d-1)}\phi_{n+1}\|_p^p.
\end{equation*}
Then it will not be difficult to see that
\begin{equation*}
	\sum_{n\in\mathbb{N}_0} 2^{n[1+p(d-1)]}\big\| \sum_{m\in\mathbb{N}_0}2^{-m(d-1)}\phi_m*\tilde{V}_m*V_n \big\|_p^p\lesssim\big\|\bigoplus_{n\in\mathbb{N}_0}\phi_n\big\|_p^p, \ \forall \bigoplus_{n\in\mathbb{N}_0}\phi_n\in L^p(\mathcal{M},\mu),
\end{equation*}
which actually proves (\ref{eq2_20}).

So it remains to prove that convolution with $V_n$ is a bounded operator whose norm does not depend on $n$. With a closer look, it only needs to prove that $\{V_n\}_{n\in\mathbb{N}_0}$ forms a sequence of uniformly bounded multipliers. This requires to split the range of $p$ in two intervals, $(0,1]$ and $(1,+\infty)$. The first case requires the extra condition of analyticity. Namely, $\{V_n\}_{n\in\mathbb{N}_0}$ should be a uniformly bounded sequence of $H^p$ multipliers. Respectively, when $p>1$, $\{V_n\}_{n\in\mathbb{N}_0}$ should be a uniformly bounded sequence of $L^p$ multipliers.

Regarding the case where $p\in(0,1]$, notice that the analyticity condition does not cause any harm, since, for every $n\in\mathbb{N}$, $V_n$ acts (as a multiplier) on the analytic projection of $L^p$. So the approach with Hardy multipliers is allowed. For we use the following theorems. Theorem \ref{thm_4} can be found in \cite[Th\'{e}or\`{e}me 1]{stein1966classes} and Theorem \ref{thm5'} in \cite[Theorem 5.1]{chen1998multiplier}
\begin{thm}\label{thm_4}
	Let $p\in(0,1]$ and consider the Hardy space $H^p(\mathbb{R})$. Let $k\in\mathbb{N}$ such that $k^{-1}<p$ and $\rho:\mathbb{R}_+\to\mathbb{C}$ which satisfies the following conditions:
	\begin{enumerate}[(i)]
		\item $|\rho(t)|\leq A, \ \forall t\in\mathbb{R}_+$;
		\item $\rho\in C^k(\mathbb{R}_+)$ and
		\begin{equation*}
			\int_R^{2R}|\rho^{(l)}(t)|^2\dd{t}\leq A R^{-2l+1}, \ \ \forall R>0, \ \forall l=1,\dots,k;
		\end{equation*}
	\end{enumerate}
	where $A$ is a positive constant. Then $\rho$ is a multiplier on $H^p(\mathbb{R})$.
\end{thm}

\begin{thm}\label{thm5'}
	Let $\rho:\mathbb{R}\to\mathbb{C}$ be a continuous function that gives rise to a multiplier on the Hardy space $H^p(\mathbb{R})$, for some $p\in(0,1]$. Then, for any $t>0$, the sequence $\rho_t=\{\rho(tn)\}_{n\in\mathbb{Z}}$ is a multiplier on the Hardy space $H^p(\mathbb{T})$ and furthermore, the multiplier norm $\left\|\rho_t\right\|_{\mathscr{M}(H^p(\mathbb{T}))}$ is uniformly bounded with respect to $t$.
\end{thm}
\noindent Notice that the polynomials $V_n$ are constructed by scaling the function $v$ (see \S\ref{BC} for definitions). Therefore, according to Theorem \ref{thm5'}, it is enough to prove that $v$ defines an $H^p(\mathbb{R})$ multiplier. The latter is achieved by applying Theorem \ref{thm_4}.

Finally, if $p>1$, then the following theorems are needed. Theorem \ref{thm4} can be found in \cite[Theorem 6.1.6]{bergh2012interpolation}, Theorem \ref{thm5} in \cite[Theorem 6.1.3]{bergh2012interpolation} and Theorem \ref{thm_5} in \cite[Theorem 4.3.7]{9781493911936}.
\begin{thm}[Mikhlin's Multiplier Theorem]\label{thm4}
	Let $\rho:\mathbb{R}\to\mathbb{C}$ be a function which satisfies
	\begin{equation*}
		|\rho^{(n)}(x)|\leq A\left<x\right>^{-n}, \ \forall x\in\mathbb{R}, \ n=0,1.
	\end{equation*}
	Then $\rho\in\mathscr{M}_p(\mathbb{R})$, for every $p\in(1,+\infty)$, and, more precisely, there exists a positive constant $C_p$ which depends only on $p$ such that
	\begin{equation*}
		\left\|\rho\right\|_{\mathscr{M}_p}\leq C_p A.
	\end{equation*}
\end{thm}

\begin{thm}\label{thm5}
	Let $\rho:\mathbb{R}\to\mathbb{C}$ belong to $\mathscr{M}_p(\mathbb{R})$. Then, for any $t\in\mathbb{R}\setminus\{0\}$, the function $\rho_t:\mathbb{R}\to\mathbb{C}$ which maps $x$ to $\rho(tx)$ belongs to $\mathscr{M}_p(\mathbb{R})$ with $\left\|\rho_t\right\|_{\mathscr{M}_p}\leq\left\|\rho\right\|_{\mathscr{M}_p}$.
\end{thm}

\begin{thm}\label{thm_5}
	Let $\rho:\mathbb{R}\to\mathbb{C}$ be a continuous function such that $\rho\in\mathscr{M}_p(\mathbb{R})$, for some $p\in(1,+\infty)$. Then, for any $t>0$, the sequence $\rho_t=\{\rho(tn)\}_{n\in\mathbb{Z}}$ belongs to $\mathscr{M}_p(\mathbb{T})$ and moreover,
	\begin{equation*}
		\sup_{t>0}\left\|\rho_t\right\|_{\mathscr{M}_p(\mathbb{T})}\leq\left\|\rho\right\|_{\mathscr{M}_p(\mathbb{R})}.
	\end{equation*}
\end{thm}
\noindent According to Theorem \ref{thm5}, it is enough to prove that $v$ gives rise to an $L^p(\mathbb{R})$ multiplier. Then Theorem \ref{thm_5} will give the desired result. The fact that $v$ defines an $L^p(\mathbb{R})$ multiplier can be checked by Theorem \ref{thm4}.

\section*{Acknowledgement}

The author is deeply indebted to A. Pushnitski for his valuable advice. His insight on the subject proved a significant guide that yielded the current results and triggered questions for further research.

\end{document}